\documentclass{article}%

\usepackage{amsmath}%
\usepackage{amsfonts}%
\usepackage{amssymb}%
\usepackage{graphicx}
\usepackage{fullpage}
\usepackage[english]{babel}
\usepackage{tikz}
\usepackage{subcaption}
\usepackage[margin=2cm]{caption}
\usepackage{dsfont}
\usepackage[bookmarksnumbered=true]{hyperref}

\usepackage{ntheorem}

\newtheorem{theorem}{Theorem}
\newtheorem*{acknowledgement}[theorem]{Acknowledgement}

\newtheorem{example}[theorem]{Example}

\newtheorem{lemma}[theorem]{Lemma}

\newtheorem{proposition}[theorem]{Proposition}

\newenvironment{proof}[1][Proof]{\textbf{#1.} }{\ \rule{0.5em}{0.5em}}

\usepackage{amsmath}%
\usepackage{amsfonts}%
\usepackage{amssymb}%
\usepackage{graphicx}
\usepackage{booktabs}
\usepackage{bbm}
\usepackage[english]{babel}
\usepackage{tikz}
\usepackage{verbatim}
\usetikzlibrary{plotmarks}
\tikzset{
    scale plot marks/.is choice,
    scale plot marks/false/.code={
        \def\pgfuseplotmark##1{\pgftransformresetnontranslations\csname pgf@plot@mark@##1\endcsname}
    },
    scale plot marks/true/.style={},
    scale plot marks/.default=true
}
\usepackage[ruled, linesnumbered]{algorithm2e}
\SetAlCapSkip{0.5em}

\usepackage{dsfont}
\usepackage[bookmarksnumbered=true]{hyperref}

\newcommand{\A}{A^N}
\newcommand{\U}{U_{(i)}^N}

\newcommand{\Corder}{{C}_{(i)}^N}

\newcommand{\LoadHat}{l^N(i)}

\newcommand{\Umn}{U_{(M)}^N}
\newcommand{\Uin}{U_{(i)}^N}
\newcommand{\Uimeen}{U_{(i)}^{M-1}}
\newcommand{\Uinm}{U_{(i)}^{N-M}}
\newcommand{\Prob}{\mathbb{P}}
\newcommand{\E}{\mathbb{E}}

\title{Robustness of Power-law Behavior\\
in Cascading Failure Models}

\author{F.~Sloothaak\footnote{F.~Sloothaak is with Eindhoven University of Technology, P.O.~Box 513, 5600 MB Eindhoven, The Netherlands.} \and S.C.~Borst\footnote{S.C.~Borst is with Nokia Bell Labs, P.O.~Box 636, Murray Hill, NJ 07974-0636, USA, and Eindhoven University of Technology, P.O.~Box 513, 5600 MB Eindhoven, The Netherlands.} \and A.P.~Zwart\footnote{A.P.~Zwart is with Centrum Wiskunde \& Informatica (CWI), P.O.~Box~94079, 1090 GB Amsterdam, The Netherlands, and Eindhoven University of Technology, P.O.~Box 513, 5600 MB Eindhoven, The Netherlands.}}

\begin{document}
\maketitle


\begin{abstract}
Inspired by reliability issues in electric transmission networks, we use a probabilistic approach to study the occurrence of large failures in a stylized cascading failure model. In this model, lines have random capacities that initially meet the load demands imposed on the network. Every single line failure changes the load distribution in the surviving network, possibly causing further lines to become overloaded and trip as well. An initial single line failure can therefore potentially trigger massive cascading effects, and in this paper we measure the risk of such cascading events by the probability that the number of failed lines exceeds a certain large threshold. Under particular critical conditions, the exceedance probability follows a power-law distribution, implying a significant risk of severe failures. We examine the robustness of the power-law behavior by exploring under which assumptions this behavior prevails.
\end{abstract}

\section{Introduction}
\label{sec:Introduction}
Cascading failure models are used to describe networks where load demands are imposed on the lines, and lines fail when their capacities cannot meet the demand. Each line failure induces changes in the load distribution in the surviving network, possibly causing further lines to trip in succession and triggering knock-on effects. Despite the admittedly stylized nature, these models capture essential features of failure processes in many settings. The abstract nature allows for a wide range of applications, such as material science, traffic networks and earthquake dynamics~\cite{Pradhan2010}. 

Our inspiration is specifically drawn from energy networks. Large blackouts of electric power transmission systems have catastrophic consequences in modern-day society. Examples include the Northeast Blackout of 2003, the India Blackout of 2012 and the Turkey Blackout of 2015. The analysis of severe blackouts has therefore become a crucial part of transmission grid planning and operations~\cite{Wang2015}. Cascading failure is a key mechanism, and typically the cascading phenomenon involves long and quite complex sequences of line failures, making the evaluation of the failure propagation difficult. 

Historically, empirical data analyses of large blackouts in North America show that the blackout size is heavy-tailed and has a power-law dependence~\cite{Dobson2004}. Work in~\cite{Carreras2002,Chen2005,Nedic2006} suggests that there is a critical loading regime where the blackout size follows a power-law distribution. This means that the probability that the blackout size is of size $k$ decreases only proportional to $k^{-\alpha}$ for a constant $\alpha>0$, which is much smaller than the exponential rate of decay for light-tailed distributions. This heavy-tailed property reflects a significant risk of seeing large blackouts occurring.

A possible method for analyzing cascading failure models involves a rare-event simulation methodology, such as importance sampling and splitting~\cite{Shortle2013,Wang2015}. A significant advantage of this technique is that it allows an analysis of fairly complex cascading failure models. However, this methodology does not provide structural insights in the mechanism leading to power-law behavior. Alternatively, Dobson et al.~\cite{Dobson2004,Dobson2005} present an analytically tractable model that shows power-law behavior for the blackout size under critical conditions. 

Our framework concerns a stochastic load-dependent cascading failure model, which can be seen as a generalization of the model of Dobson et al.~\cite{Dobson2005}. Specifically, we consider a network where the load is carried on $N$ statistically identical lines. We assume that the system is initially stable, and lines have stochastic surplus capacities. An initiating disturbance causes an increased load at all lines, which may lead to a cascading failure process when the load surge exceeds the surplus capacity of one or more lines. In order to understand and quantify the risks of cascading failures, we study the probability that the number of failed lines exceeds some threshold $k$.

The model presented by Dobson et al.~\cite{Dobson2005} is a special case of our framework. This model assumes uniformly distributed surplus capacities and a load surge function that is linearly increasing. It turns out that the blackout size follows a quasi-binomial distribution in this case, which converges to a generalized Poisson distribution when $N \rightarrow \infty$. This is the same distribution as the number of offspring in a branching process. In~\cite{Dobson2004}, Dobson et al.~use this branching process as an approximation for the probability that the blackout size is exactly $k$. For a particular critical regime, it is known~\cite{Otter1949} that this yields power-law behavior with coefficient $3/2$ when $k \rightarrow \infty$. 

In the present paper we study the robustness of the power-law behavior for the model of~\cite{Dobson2005}, and extend these results in two directions. First, we investigate whether the power-law behavior prevails when the threshold $k$ depends on the network size $N$ under similar assumptions on the surplus capacities and the load surge function as in~\cite{Dobson2005}. This extension provides a rigorous justification for approximations such as the probability that the network partially breaks down, e.g.~the probability that at least $20 \%$ of all lines fail. We show that this extension still results in power-law behavior, and observe that in case $k= \alpha N$ for some $\alpha \in (0,1)$, the corresponding prefactor differs from the branching process approximation in~\cite{Dobson2004}. Our result indicates that the dependency between $k$ and $N$ must be accounted for in order to derive an asymptotically exact result.

Second, we investigate under which broader assumptions on the surplus capacities and the load surge function in~\cite{Dobson2005} the power-law distribution of the blackout size is preserved. It turns out that this can be captured by identifying the possible functions for the composition of the surplus capacity distribution function and the load surge function. When this composition is a linearly increasing function with critical slope, we obtain the special case of Dobson et al~\cite{Dobson2005}. We show that the power-law behavior also prevails for settings where the linearly increasing load surge function is perturbed under specific conditions. Whether these conditions are satisfied depends on three factors, namely the surplus capacity distribution, the load surge function and the threshold $k$. We conclude by considering examples with given surplus capacity distribution and load surge function, and identify the possible thresholds $k$ that yield power-law behavior.

The rest of this paper is organized as follows. In Section~\ref{sec:ModelDescription} we describe the cascading failure model. We explain our main results in detail in Section~\ref{sec:MainResults}, and present the proofs in Sections~\ref{sec:ProofsOfB} and \ref{sec:ProofsOfC}. In Section~\ref{sec:IdentifyCriticalRegion} we consider illustrative examples that identify thresholds $k$ where the power-law behavior prevails. We present a few concluding remarks and discuss possible directions for further research in Section~\ref{sec:Outlook}.

\section{Model description}
\label{sec:ModelDescription}
We consider a network consisting of $N$ statistically identical lines. Each line has a limited capacity for the amount of load it can carry before it trips. We assume that the network is initially stable, i.e.~all lines have capacities that exceed their initial load. The difference between the initial load and capacity of line $i$ is denoted by $C^N(i)$, and referred to as the \textit{surplus capacity}. We assume the surplus capacities of the various lines to be i.i.d.~with common distribution function $F(\cdot)$. 

In order to trigger a possible cascading failure effect, we include a failing dummy line in the network. The failing dummy line causes an initial increase in load for the surviving lines, which can then possibly fail as well. We denote by $l^N(i)$ the \textit{load surge} on each surviving line when the dummy line plus $i-1$ other lines have failed, and assume this is a deterministic function. Since line failures cause additional load on the remaining lines, $l^N(\cdot)$ is thus an increasing function. 

The main objective in this paper involves the probability that $\A$, the number of failed lines in the network, exceeds a certain threshold $k$ as $N$ grows large. In order to express this in mathematical terms, we take a closer look at the cascading failure process. After the dummy line has tripped, a next line will fail when the smallest surplus capacity is exceeded by the load surge $l^N(1)$. If so, another line will fail if and only if $l^N(2)$ exceeds the second smallest surplus capacity and so forth. Denote by $C_{(1)}^N \leq {C}_{(2)}^N\leq ... \leq {C}_{(N)}^N$ the ordered surplus capacities. The above observation yields that the blackout size is given by
\begin{align}
A_N = \max \{ k \in \mathbb{N} : {C}^N_{(i)} \leq {l}^N(i), \;\;\; i=1,...,k \}
\label{eq:1A}
\end{align}
and $A_N=0$ if ${C}^N_{(1)} > {l}^N(1)$. The probability that the blackout size exceeds an integer $k$ is thus given by
\begin{equation}
\Prob(\A \geq k) = \Prob\left( C^N(i) \leq l^N(i), \;\;\; i=1,...,k\right).
\label{eq:failedLinesProb}
\end{equation}

In fact, \eqref{eq:failedLinesProb}~can be rewritten into an expression that is easier to analyze. Let $\U$ denote the standard uniformly distributed ordered statistics for $i=1,...,N$ and suppose $F(\cdot)$ is continuous. Following Lemma~4.1.9 in~\cite[p.188]{Embrechts1997}, we find that $(F(\Corder))_{i=1,...,N}$ and $(\U))_{i=1,...,N}$ are equal in distribution. Therefore,~\eqref{eq:failedLinesProb} is equivalent to
\begin{equation}
\Prob(\A \geq k) = \Prob\left( \U \leq F(\LoadHat), \;\;\; i=1,...,k\right).
\label{eq:OrderedProb}
\end{equation}

\noindent
Note that $F(l^N(i))$ represents the probability that a line does not have sufficient surplus capacity to sustain the load surge after $i-1$ other lines have failed. We observe that $F \circ l^N$ is an increasing function in the number of failed lines with support $[0,1]$. In addition, we assume that the capacity surplus distribution has a strictly positive density in zero.

We note that our model does not explicitly account for many complexities that exist in real electric power transmission systems, such as the length of time between occurrences or the network topology that can lead to multiple line capacity distributions or non-equal load distribution. Yet, this model captures two important features of large blackouts: the initial disturbance loading the system and the cascading line failure mechanism.

Our framework can be seen as an extension of the model presented by Dobson et al.~\cite{Dobson2005}. Their model comprises uniformly distributed initial loads, where lines fail when a fixed load limit (larger than all possible initial loads) is exceeded. Due to the properties of the uniform distribution, we observe that the surplus capacity is in that case also uniformly distributed. To start the cascading process, there is an initial fixed load increase for all lines and each failing line induces a fixed increase of load on the surviving lines. After normalizing, we see that that this model is a special case of our framework with standard uniformly distributed surplus capacities and a load surge function of the form
\begin{align}
{l}^N_D(i) = \frac{\theta +(i-1)\lambda}{N}
\label{eq:LoadSurgeAffine}
\end{align}
for positive constants $\theta$ and $\lambda$. In other words, the failing dummy line induces a load surge of $\theta/N$ at each line, and every next failing line induces a fixed load surge of $\lambda/N$ at each surviving line. In view of~\eqref{eq:LoadSurgeAffine}, we refer to this particular setting as the affine case. Dobson et al.~indicate that when $N \rightarrow \infty$, the blackout size distribution converges to a generalized Poisson distribution, which is also the distribution of the number of offspring in a particular branching process. In~\cite{Dobson2004}, Dobson et al.~use the branching process as an approximation for the blackout size and note that $\lambda=1$ corresponds to a critical window where a power-law dependence manifests itself. For this critical window, the branching process approximation is given by
\begin{align}
\Prob(A^N=k) \approx \frac{\theta}{\sqrt{2\pi}}k^{-3/2}.
\label{eq:DobsonBranchingApproximation}
\end{align}

\section{Main results}
\label{sec:MainResults}
The main object of interest in this paper is the probability that the number of failed lines $A^N$ exceeds some threshold $k:=k(N)$ with both $k \rightarrow \infty$ and $N-k \rightarrow \infty$ as $N \rightarrow \infty$. For compactness, we suppress the dependence on $N$ in the remainder of the paper. In this section we provide an overview of the main results and implications for the robustness of the power-law behavior.

\subsection{The affine case}
We first examine the robustness of the power-law behavior of the affine model, which has uniformly distributed surplus capacities and a load surge function of the form~\eqref{eq:LoadSurgeAffine}, for thresholds $k$ growing with $N$. As a matter of fact, we note that Equation~\eqref{eq:OrderedProb} implies that the affine case covers all cases for which
\begin{align}
F({l}^N(i)) =  \frac{\theta+i-1}{N}
\label{eq:relationLoadCap}
\end{align}
holds for some constant $\theta >0$. That is, the composition $F \circ l^N$ needs to be linearly increasing with step increments $1/N$. 

This extension thus involves the same assumption as the model of Dobson et al.~\cite{Dobson2005}, but accounts for the dependence between $k$ and $N$. For such thresholds, we will obtain the approximation
\begin{align}
\Prob(A^N=k) \approx \frac{\theta}{\sqrt{2\pi}} \frac{1}{\sqrt{1-k/N}} k^{-3/2}.
\label{eq:AffineOurApproximation}
\end{align}
We observe that if $k$ is proportional to the network size, i.e.~$k=\alpha N$ for some fixed coefficient $\alpha \in (0,1)$, our approximation leads to a different prefactor than the branching process approximation~\eqref{eq:DobsonBranchingApproximation}. This difference emerges because the derivation of~\eqref{eq:DobsonBranchingApproximation} essentially uses a double limit approach. That is, it first lets $N \rightarrow \infty$ for fixed $k$, and then lets $k \rightarrow \infty$. Specifically, \eqref{eq:DobsonBranchingApproximation}~originates from the result
\begin{align}
\lim_{k \rightarrow \infty} k^{3/2} \lim_{N \rightarrow \infty} \Prob(\A=k) = \frac{\theta}{\sqrt{2\pi}},
\label{eq:DobsonApproximationFormal}
\end{align}
see Section~\ref{sec:ProofsOfB} for a proof. Our result, stated in Theorem~\ref{thm:ToBeProvedProb}, shows that for large network size dependent thresholds $k$, the power-law behavior indeed prevails, but the result needs to account for this dependence in order to be asymptotically exact.

\begin{theorem}
Let $F\circ l^N$ be as in~\eqref{eq:relationLoadCap} with constant $\theta>0$ for each $N \in \mathbb{N}$. Let $k_\star:=k_\star(N)$ and $k^\star:=k^\star(N)$ be functions of $N$ with $k^\star \geq k_\star$, $k_\star \rightarrow \infty$ and $N-k^\star \rightarrow \infty$ as $N \rightarrow \infty$. Then,
\begin{equation}
\lim_{N \rightarrow \infty} \sup_{k \in [k_\star,k^\star]} \left|k^{3/2}\sqrt{1-k/N} \Prob(\A=k)-\frac{\theta}{\sqrt{2 \pi}} \right| =0.
\label{eq:ToBeProvedEquation}
\end{equation}
\label{thm:ToBeProvedProb}
\end{theorem}

Dobson et al.~show in~\cite{Dobson2005} that the blackout size for the affine model follows a quasi-binomial distribution, and the proof of Theorem~\ref{thm:ToBeProvedProb} relies heavily on the explicit form of that distribution function. We note that the same technique can be used for fixed $k$, which yields the generalized Poisson distribution, or for $k=N-l$ with $l>\theta$ a fixed integer. This gives rise to the results summarized in Table~\ref{tab:Overview}.

\begin{table}[htb]
\centering
\begin{tabular}{ll}
\toprule
\multicolumn{2}{c}{$\Prob(A^N=k)$}  \\
\midrule
$k$ fixed & $\frac{\theta \left(\theta+k\right)^{k-1}}{k!} e^{-(\theta+k)}$ \\
$k$ growing & $\frac{\theta}{\sqrt{2 \pi}} \frac{k^{-3/2}}{\sqrt{1-k/N}} $ \\
$k=N-l$, $l$ fixed & $\frac{\theta(l-\theta)^l}{l!} e^{-(l-\theta)}N^{-1}$\\
\bottomrule
\\
\end{tabular}
\caption{Asymptotic behavior affine case}
\label{tab:Overview}
\end{table}

We can use Theorem~\ref{thm:ToBeProvedProb} to derive the asymptotic probability that the blackout size exceeds the threshold~$k$:
\begin{align}
\Prob(\A \geq k) \approx \frac{2\theta}{\sqrt{2 \pi}} \sqrt{1-k/N} k^{-1/2}.
\label{eq:AffineExceedanceOurApproximation}
\end{align}
For example, if $k=\alpha N$ with $\alpha \in (0,1)$, we observe the power-law behavior $\Prob(\A \geq k) \approx 2\theta/\sqrt{2 \pi} \sqrt{1-\alpha} k^{-1/2}$. This approximation is supported by the following result.
\begin{theorem}
Let $F\circ l^N$ be of the form~\eqref{eq:relationLoadCap} with constant $\theta>0$ for each $N \in \mathbb{N}$, and $k:=k(N) \leq N$ a positive function of $N$ such that $k \rightarrow \infty$ and $N-k \rightarrow \infty$ as $N \rightarrow \infty$. Then,
\begin{align}
\lim_{N \rightarrow \infty} \sqrt{\frac{k N}{N-k}} \Prob(\A \geq k) = \frac{2\theta}{\sqrt{2 \pi}}.
\label{eq:CumulativeProbGeneral}
\end{align}
\label{thm:CumulativeProbGeneral}
\end{theorem}
We conclude from Theorem~\ref{thm:CumulativeProbGeneral} that the power-law behavior for the affine model extends to thresholds $k$ that are appropriately growing functions of the network size.

\subsection{Perturbations of the composition}
Since the form of~~\eqref{eq:relationLoadCap} is rather fragile, we explore the robustness of the power-law behavior as observed in~\eqref{eq:AffineOurApproximation} and~\eqref{eq:AffineExceedanceOurApproximation} with respect to perturbations of~\eqref{eq:relationLoadCap}. Specifically, we consider compositions $F \circ l^N$ of the form
\begin{align}
F({l}^N(i))= \frac{\theta+i-1+\Delta(i,N)}{N},
\label{eq:FinitePerturbLoadFunction}
\end{align}
where $\Delta(\cdot,\cdot)$ represent the perturbations with respect to the corresponding affine case. We make suitable assumptions on the magnitude of the perturbations such that the power-law behavior prevails. Intuitively, this means that the perturbations $\Delta(i,N)$ are small for large values of $i \leq k$ and $N$. The exact conditions, specified in Section~\ref{sec:ProofsOfC}, are technical and therefore omitted here. 

\begin{theorem}
Let $k:=k(N) \leq N$ be a positive function of $N$ such that $k \rightarrow \infty$ and $N-k \rightarrow \infty$ as $N \rightarrow \infty$, and let $F\circ l^N$ be as in~\eqref{eq:FinitePerturbLoadFunction} with $\Delta(\cdot,\cdot)$ satisfying properties (A)-(D) defined in Section~\ref{sec:ProofsOfC}. Then, there exists a constant $V(\theta,\Delta) \in (0,\infty)$ such that
\begin{align}
\lim_{N \rightarrow \infty} \sqrt{\frac{kN}{N-k}} \Prob(A^N \geq k) = V(\theta,\Delta).
\end{align}
\label{thm:PerturbConstant}
\end{theorem}

That is, if the conditions on $\Delta(\cdot,\cdot)$ specified in Section~\ref{sec:ProofsOfC} are satisfied, there exists a finite, strictly positive constant $V(\theta,\Delta)$ (not depending on $k$) such that
\begin{align}
\Prob(A^N \geq k) \approx V(\theta,\Delta) \sqrt{\frac{N-k}{kN}}.
\label{eq:PerturbationAproximation}
\end{align}

The constant $V(\theta,\Delta)$, also defined in Section~\ref{sec:ProofsOfC}, is generally difficult to compute explicitly, but we can approximate its value with arbitrary precision. We present an algorithm for this in Section~\ref{sec:ProofsOfC}.

Whether the assumptions on the perturbations $\Delta(\cdot,\cdot)$ are satisfied, depends on the surplus capacity distribution $F(\cdot)$, the load surge function $l^N(\cdot)$ and the threshold $k$. In Section~\ref{sec:IdentifyCriticalRegion} we consider two examples where we identify thresholds $k$ such that the power-law behavior prevails. The most compelling example involves the case where the loads of the failed lines are equally redistributed over the remaining lines, see for example the model presented by Shortle~\cite{Shortle2013}. Suppose that each line has an initial load $a$, and hence the network has a total load of $a N$. Every time a line fails, the total load is redistributed over all surviving lines. The load surge function is then given by
\begin{align}
l^N(i) = \frac{a i}{N-i}.
\label{eq:1}
\end{align} 
Theorem~\ref{thm:PerturbConstant} suggests that when the step increments of the composition $F\circ l^N$ are approximately $1/N$ up to the $k$'th failure, the power-law behavior prevails. A Taylor expansion of the composition $F\circ l^N$ suggests that this occurs when $a F'(0)=1$. Yet, we observe that the slope of the load surge function increases as $i$ grows. For example, the slope is approximately $a/N$ for fixed $i$ and approximately $(a/(1-\alpha))/N > a /N$ for $i= \alpha N$ for some $\alpha \in (0,1)$. The threshold $k$ should thus be of a size such that no difference in slope is observable for the composition $F \circ l^N$ up to the threshold $k$. As one might expect, if the surplus capacity is uniformly distributed, then all thresholds $k=o(N)$ meet this condition. In Example~\ref{ex:Shortle} in Section~\ref{sec:IdentifyCriticalRegion}, we explain that generally all thresholds $k= o(\sqrt{N})$ lead to power-law behavior.

\section{Proofs for the affine case}
\label{sec:ProofsOfB}
When relation~\eqref{eq:relationLoadCap} holds, the blackout size follows a quasi-binomial distribution~\cite{Dobson2005}. This ensures an analytic expression for the probability distribution of the blackout, which we use to derive the asymptotic behavior.

\begin{proof}[Proof of Theorem~\ref{thm:ToBeProvedProb}]
The probability distribution of the number of failed lines is given by~\cite{Dobson2005}
\begin{align}
\Prob(\A=k)
&= \left\{ \begin{array}{ll}
\binom{N}{k} \frac{\theta}{N} \left( \frac{\theta+k}{N} \right)^{k-1} \left(1-\frac{k+\theta}{N}\right)^{N-k}, & \textrm{if } k \leq N-\theta,\\
0,  & \textrm{if } N-\theta < k < N, \\
\sum_{i=\lfloor N-\theta\rfloor+1}^N \binom{N}{i} \frac{\theta}{N} \left( \frac{\theta+i}{N} \right)^{i-1} \left(1-\frac{i+\theta}{N}\right)^{N-i},& \textrm{if } k = N.
\end{array}\right.
\label{eq:FailedLinesProb}
\end{align}

\noindent
Since we consider $k \in [k_\star, k^\star]$, we are only concerned with the probability for $k \leq N-\theta$. With Stirling's approximation~(formula (6.1.38) of~\cite{Abramowitz1964}) we have that for every integer $m > 0$
\[
m! = \sqrt{2 \pi} m^{m+1/2} e^{-m+\frac{y(m)}{12m}},
\]
for some $0< y(m) <1$. So the binomial term is bounded by
\begin{align*}
\binom{N}{k} &\geq \frac{1}{\sqrt{2\pi}} \frac{N^N}{k^k(N-k)^{N-k}} \sqrt{\frac{N}{k(N-k)}} e^{-\frac{1}{12k}} e^{-\frac{1}{12(N-k)}}, \\
\binom{N}{k} &\leq \frac{1}{\sqrt{2\pi}} \frac{N^N}{k^k(N-k)^{N-k}} \sqrt{\frac{N}{k(N-k)}} e^{\frac{1}{12N}}.
\end{align*} 

\noindent
Using these bounds for the binomial term in~\eqref{eq:FailedLinesProb} yields
\begin{align*}
k^{3/2}\sqrt{1-k/N} \Prob(\A=k) &\geq \frac{\theta}{\sqrt{2\pi}} \left( 1+\frac{\theta}{k} \right)^{k-1} \left(1-\frac{\theta}{N-k}\right)^{N-k} e^{-\frac{1}{12k}} e^{-\frac{1}{12(N-k)}}
\end{align*}
and
\begin{align*}
k^{3/2}&\sqrt{1-k/N} \Prob(\A=k) \leq \frac{\theta}{\sqrt{2\pi}} \left( 1+\frac{\theta}{k} \right)^{k-1} \left(1-\frac{\theta}{N-k}\right)^{N-k} e^{\frac{1}{12N}}
\end{align*}
for any $k_\star \leq k \leq k^\star$. Note that for every constant $\theta >0$ the functions $(1+\theta/x)^{x-1}$ and $(1-\theta/x)^{x}$ are both monotone increasing in $x>0$. Moreover, the function $e^{-1/(12 x)}$ is monotone increasing in $x>0$. Therefore, we obtain the lower bound
\begin{align*}
\sup_{k \in [k_\star,k^\star]} k^{3/2}\sqrt{1-k/N} \Prob(\A=k) & \geq \frac{\theta}{\sqrt{2\pi}}  \sup_{k \in [k_\star,k^\star]}  \left( 1+\frac{\theta}{k} \right)^{k-1} e^{-\frac{1}{12k}} \cdot \sup_{k \in [k_\star,k^\star]}  \left(1-\frac{\theta}{N-k}\right)^{N-k} e^{-\frac{1}{12(N-k)}}\\
&= \frac{\theta}{\sqrt{2\pi}} \left( 1+\frac{\theta}{k^\star} \right)^{k^\star-1} \left(1-\frac{\theta}{N-k_\star}\right)^{N-k_\star} e^{-\frac{1}{12k^\star}} e^{-\frac{1}{12(N-k_\star)}}.
\end{align*}
Moreover, since $(1+\theta/x)^{x-1} \leq e^\theta$ and $(1-\theta/x)^{x} \leq e^{-\theta}$ for all $x>0$, we have the upper bound
\begin{align*}
\sup_{k \in [k_\star,k^\star]} k^{3/2}\sqrt{1-k/N} \Prob(\A=k) &\leq \frac{\theta}{\sqrt{2\pi}} e^{\frac{1}{12N}}.
\end{align*}

We observe that both the upper bound and the lower bound converge to $\theta/\sqrt{2\pi}$ as $N \rightarrow \infty$ under the given assumptions on $k_\star$ and $k^\star$, implying that~\eqref{eq:ToBeProvedEquation} holds.
\end{proof}

\begin{figure}[htb]
\centering
\begin{tikzpicture}[xscale=0.4, yscale=1.6]
\draw[->,thick] (0,0) -- (0,3);
\draw[->, thick] (0,0) -- (21,0);

\draw plot file {Stair.txt};
\draw[red,dashed] plot[smooth] file {UBAs.txt};
\draw[blue,dashed] plot[smooth] file {LBAs.txt};

\draw [thick] (10.5,-0.05) -- (10.5,0.05);
\draw node[below] at (10.5,0) {$i=0.75N$};
\draw node[below] at (20,0) {$i$};
\end{tikzpicture}
\caption{Continuous bounds for $\Prob(A^N=i)$.}
\label{fig:FunctionIllustration}
\end{figure}
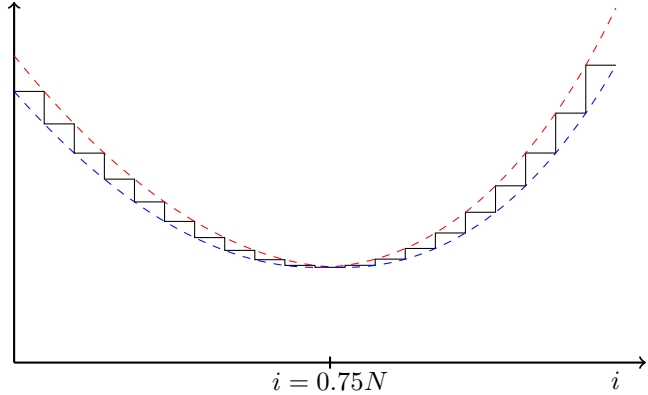

We indicated in Section~\ref{sec:MainResults} that the branching process approximation yields a different prefactor when $k=\alpha N$ for some $\alpha \in (0,1)$. In order to justify this claim we prove~\eqref{eq:DobsonApproximationFormal} next.

\begin{proof}[Proof of~\eqref{eq:DobsonApproximationFormal}]
The blackout size converges in distribution to a generalized Poisson distribution, i.e.~\cite{Dobson2005}
\begin{align*}
\lim_{N \rightarrow \infty} \Prob(A^N=k) = \theta \frac{(\theta+k)^{k-1}}{k!}e^{-(\theta+k)}.
\end{align*}

\newpage
\noindent
Applying Stirling's approximation, we obtain
\begin{align*}
\lim_{k \rightarrow \infty} k^{3/2} \lim_{N \rightarrow \infty} \Prob(A^N=k) &= \lim_{k \rightarrow \infty} k^{3/2} \theta \frac{(\theta+k)^{k-1}}{k!}e^{-(\theta+k)} \\
&= \lim_{k \rightarrow \infty} \frac{\theta}{\sqrt{2\pi}} \left(1+\frac{\theta}{k} \right)^{k-1}e^{-\theta} = \frac{\theta}{\sqrt{2 \pi}}.
\end{align*}
\end{proof}

Next, we turn to the asymptotic behavior of the probability that the blackout size exceeds the threshold $k$. For this, we bound the discrete density function of the blackout size by two continuous functions that grow arbitrarily close to one another for all $i \in [k,N-\log(N-k)]$. We conclude the proof by deriving the integral counterparts of the continuous functions and showing that $\Prob(A^N \geq N-\log(N-k))$ is asymptotically negligible.

\noindent
\begin{proof}[Proof of Theorem~\ref{thm:CumulativeProbGeneral}]
Set $k^\star = N-\log(N-k)$. Observe that $k^\star \geq k$ and note that this choice ensures $k^\star = N-o(N)$ and $N-k^\star \rightarrow \infty$ as $N \rightarrow \infty$. Due to Theorem~\ref{thm:ToBeProvedProb}, it follows that for every $\epsilon >0$ there is a $N_\epsilon >0$ such that for all $N \geq N_\epsilon$ and $i \in [k,k^\star]$ 
\begin{align*}
\frac{\theta}{\sqrt{2 \pi}} &i^{-3/2} (1-i/N)^{-1/2}(1-\epsilon) \leq \Prob(\A = i) \leq \frac{\theta}{\sqrt{2 \pi}} i^{-3/2} (1-i/N)^{-1/2}(1+\epsilon).
\end{align*}
Next, we use this observation to bound the exceedance probability from above and below and show that these bounds coincide as $\epsilon \downarrow 0$. 

An upper bound for the exceedance probability is given by
\begin{align*}
\Prob(\A \geq k) &\leq \Prob(A^N = k) +(1+\epsilon) \sum_{i=k+1}^{k^{\star}} \frac{\theta}{\sqrt{2 \pi}} i^{-3/2} (1-i/N)^{-1/2} +  \sum_{i=k^{\star}+1}^N \Prob(\A = i). 
\end{align*}

We indicate that we consider the first term separately from the second term, because this results in a nicer expression for the second term and the contribution of $\Prob(A^N = k) $ is asymptotically negligible. That is, for every integer $m$ Stirling's bound~\cite{Abramowitz1964} yields $ \sqrt{2 \pi} m^{m+1/2} e^{-m} \leq m! \leq e \sqrt{2 \pi} m^{m+1/2} e^{-m}$. Therefore,
\begin{align*}
\sqrt{\frac{kN}{N-k}} \Prob(A^N = k) & \leq e \sqrt{\frac{N}{N-k}} \sqrt{\frac{N}{N-k}}  \frac{\theta}{\theta + k}  \left( 1+ \frac{\theta}{k}\right)^{k} \left( 1- \frac{\theta}{N-k}\right)^{N-k} \\
& \leq e \frac{N}{k(N-k)} \frac{\theta}{\theta/k + 1} \longrightarrow 0
\end{align*}
as $N \rightarrow \infty$. 

For the second term, we consider the integral
\begin{align*}
\int_{x=k}^{N} x^{-3/2}\left(1-\frac{x}{N}\right)^{-1/2} \, dx &= \int_{u=\arcsin(\sqrt{k/N})}^{\pi/2} N^{-3/2} \sin(u)^{-3} (1-\sin(u)^2)^{-1/2} \cdot 2N \sin(u) \cos(u) \,du \\
&= 2 N^{-1/2} \int_{u=\arcsin(\sqrt{k/N})}^{\pi/2} \sin(u)^{-2} \,du \\
&= 2 N^{-1/2} \frac{\sqrt{1-k/N}}{\sqrt{k/N}} = 2\sqrt{\frac{N-k}{kN}},
\end{align*}
where we applied the variable substitution $x=N \sin(u)^2$. Then, the second term is bounded by
\begin{align*}
\sum_{i=k+1}^{k^{\star}} &\frac{\theta}{\sqrt{2 \pi}} i^{-3/2} (1-i/N)^{-1/2} \leq \int_{i=k}^{N} \frac{\theta}{\sqrt{2 \pi}} i^{-3/2} (1-i/N)^{-1/2} \, di = \frac{2\theta}{\sqrt{2\pi}} \sqrt{\frac{N-k}{kN}}.
\end{align*}

\noindent
Finally,
\begin{align*}
\sum_{i=k^{\star}+1}^N \Prob(\A = i) \leq (N-k^\star) \sup_{i \in (k^\star,N]} \Prob(A^N =i).
\end{align*}
To determine the supremum, we take a closer look at~\eqref{eq:FailedLinesProb}. For all $i \in (N-\theta, N)$, if any, $\Prob(A^N=i)=0$. Moreover, for all integers $i \in (k^\star, N-1]$, Stirling's bound yields
\begin{align*}
\binom{N}{i}\frac{\theta}{N}\left(\frac{\theta+i}{N}\right)^{i-1}\left(1-\frac{\theta+i}{N}\right)^{N-i} 
&\leq e \sqrt{\frac{N}{i(N-i)}} \frac{\theta}{\theta+i} \left(1+\frac{\theta}{i}\right)^{i}\left(1-\frac{\theta}{N-i}\right)^{N-i} \\
&\leq e \sqrt{\frac{N}{N-1}} \frac{\theta}{\theta+k^\star}.
\end{align*}
Therefore, $\sup_{i \in (k^\star, N-\theta]} \Prob(A^N=i) \leq c_1/k^\star$ for some constant $c_1>0$, and 
\begin{align*}
\Prob(A^N =N) &= \frac{\theta}{N} \left(1+\frac{\theta}{N} \right)^{N-1} + \sum_{i=\lfloor N-\theta\rfloor+1}^{N-1} \binom{N}{i} \frac{\theta}{N} \left( \frac{\theta+i}{N} \right)^{i-1} \left(1-\frac{\theta+i}{N}\right)^{N-i} \\
&\leq \frac{\theta e^{\theta}}{N} + \theta e \sqrt{\frac{N}{N-1}} \frac{\theta}{\theta+k^\star} \leq \frac{c_2}{k^\star}
\end{align*}
for some constant $c_2>0$. Recall that $k \leq k^\star = N-\log(N-k)$. Setting $c=\max\{c_1,c_2\}$ yields
\begin{align*}
\sqrt{\frac{kN}{N-k}}  (N-k^\star) \sup_{i \in (k^\star,N]} \Prob(A^N =i) \leq c \frac{\sqrt{kN}}{k^\star}\frac{N-k^\star}{\sqrt{N-k}} = c \underbrace{\frac{\sqrt{k/N}}{1-\log(N-k)/N}}_{=O(1)} \underbrace{ \frac{\log(N-k)}{\sqrt{N-k}}}_{= o(1)}
\end{align*}
as $N \rightarrow \infty$, since $N-k \rightarrow \infty$ as $N \rightarrow \infty$. We conclude that
\begin{align*}
\limsup_{N \rightarrow \infty}  \sqrt{\frac{k N}{N-k}} \Prob(\A \geq k) \leq (1+\epsilon) \frac{2\theta}{\sqrt{2 \pi}}.
\end{align*}

\noindent
A lower bound is given by
\begin{align*}
\Prob(\A \geq k) &\geq (1-\epsilon) \sum_{i=k}^{k^{\star}} \frac{\theta}{\sqrt{2 \pi}} i^{-3/2} (1-i/N)^{-1/2} \geq  (1-\epsilon) \int_{i=k}^{k^\star} \frac{\theta}{\sqrt{2 \pi}} i^{-3/2} (1-i/N)^{-1/2} \, di \\
&= (1-\epsilon) \frac{2\theta}{\sqrt{2\pi}}\left( \sqrt{\frac{N-k}{kN}}- \sqrt{\frac{N-k^\star}{k^\star N}} \right).
\end{align*}
It follows that
\begin{align*}
\liminf_{N \rightarrow \infty} &\sqrt{\frac{kN}{N-k}} \Prob(\A \geq k) \geq \liminf_{N \rightarrow \infty} (1-\epsilon) \frac{2\theta}{\sqrt{2\pi}}\left( 1- \underbrace{\sqrt{\frac{k}{k^\star}}}_{=O(1)} \underbrace{\sqrt{\frac{N-k^\star}{N-k}}}_{=o(1)} \right) = (1-\epsilon)\frac{2\theta}{\sqrt{2\pi}}.
\end{align*}
As $\epsilon \downarrow 0$, the limsup and liminf coincide.
\end{proof}

\section{Proofs for perturbations of the composition}
\label{sec:ProofsOfC}
Whether we obtain power-law behavior for the black-out size distribution depends on the surplus capacity distribution, the load surge function and the threshold $k$. Due to relation~\eqref{eq:OrderedProb}, we observe that the relation between the surplus capacity distribution and the load surge function is captured by the composition $F\circ l^N$, see Figure~\ref{fig:InteractionExplanation}. In this section we prove that if $F\circ l^N$ has a form as in~\eqref{eq:FinitePerturbLoadFunction}, with suitable conditions on the perturbations $\Delta(\cdot,\cdot)$, the power-law behavior for the exceedance probability prevails. 

\begin{figure}[htb]
\centering
\begin{subfigure}{0.45 \textwidth}
\centering
\scalebox{0.8}{\begin{tikzpicture}[xscale=2, yscale=4]
\draw[->] (0,0) -- (0,1.2);
\draw[->] (0,0) -- (4.2,0);

\draw plot[smooth] file {test2.txt};
\draw plot file {trap2.txt};

\draw node[] at (3.8181818,2.51-1.65) {$F(x)$};
\draw node[below] at (4.047,0) {$x$};
\draw node[below] at (2.2092,0) {${l}^N(k)$};

\foreach \i in {0., 0.0391694, 0.0804575, 0.124105, 0.170401, 0.219684, 0.27237, 0.328962, 0.390085, 0.45653, 0.529312, 0.60977, 0.699713, 0.801681, 0.919395, 1.05862, 1.22903, 1.44871, 1.75834, 2.28765 , 3}
{
\draw (\i,-0.025) -- (\i,0.025);
}

\end{tikzpicture}}
\end{subfigure}
\begin{subfigure}{0.45 \textwidth}
\centering
\scalebox{0.8}{\begin{tikzpicture}[xscale=2, yscale=4]
\draw[->] (0,0) -- (0,1.2);
\draw[->] (0,0) -- (4.2,0);

\draw[] plot[smooth] file {ThetaLine.txt};
\draw plot[smooth] file {InfinitePerturb2.txt};

\draw[] (2.65,-0.02) -- (2.65,0.02);
\draw node[below] at (2.65,0) {$k$};

\draw[<->] (0.30,0.26) -- (0.30,0.17);
\draw node[below right] at (0.30,0.17) {$\Delta(i,N)$};

\draw node[below] at (4,0) {$i$};
\draw node[below] at (4,0.9) {$F(l^N(i))$};

\end{tikzpicture}}
\end{subfigure}
\caption{Relation surplus capacity distribution function and load surge function.}
\label{fig:InteractionExplanation}
\end{figure}
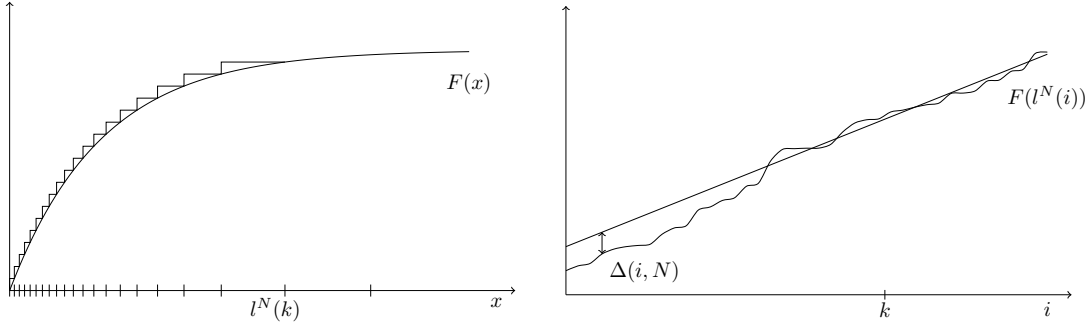

What are the suitable conditions on the perturbations? In view of~\eqref{eq:FinitePerturbLoadFunction}, we note that the perturbations $\Delta(\cdot,\cdot)$ already satisfy two properties because $F \circ l^N$ is linearly increasing with support $[0,1]$, namely
\begin{itemize}
\item[(A)] $\Delta(i,N) \in [-(\theta+i-1),N-(\theta+i-1)]$ for every $(i,N) \in \mathbb{N} \times \mathbb{N}$,
\item[(B)] for every fixed $N \in \mathbb{N}$, $\Delta(i+1,N) \geq \Delta(i,N)-1$ for all $i \geq \mathbb{N}$.
\end{itemize}
In addition, we require that the perturbations $\Delta(i,N)$ are small for large values of $i$ and $N$. Formally, we assume that
\begin{itemize}
\item[(C)] $\Delta(i,N) \rightarrow \Delta(i)$ pointwise as $N \rightarrow \infty$ for some well-defined function $\Delta(\cdot): \mathbb{N} \rightarrow \mathbb{R}$ satisfying $\lim_{i \rightarrow \infty}\Delta(i)=0$.
\item[(D)] For all $i(N) \leq k$ satisfying $\lim_{N \rightarrow \infty} i(N) = \infty$, it must hold that $\lim_{N \rightarrow \infty} \Delta(i(N),N) = 0$.
\end{itemize}

\noindent
Write $c_{i,N}= N F(l^N(i)) = \theta+i-1+\Delta(i,N)$ and for every fixed $i \in \mathbb{N}$, 
\begin{align}
c_{i} = \lim_{N \rightarrow \infty} N F(l^N(i)) = \theta+i-1+\Delta(i).
\label{eq:DefinitionCAss}
\end{align}
Note that by conditions~(A)-(D), $c_i$, $i \in \mathbb{N}$, is a well-defined non-decreasing sequence that tends to the line $\theta+i-1$ as $i$ grows large. We will show that these conditions result in power-law behavior for the exceedance probability. 

For this, we leverage two basic asymptotic properties formulated in the following two lemmas.
\begin{lemma}
Let $k$ be a function of $N$ such that both $k \rightarrow \infty$ and $N-k \rightarrow \infty$ as $N \rightarrow \infty$. Then for every fixed $M_1, M_2 \in~\mathbb{N}$,
\begin{align*}
\lim_{N \rightarrow \infty} \sqrt{\frac{kN}{N-k}} \Prob&\left( U_{(i)}^{N-M_1} \leq \frac{\theta+i-1}{N-M_1},  \;\;\; \forall i\leq k-M_2 \right) \\
&= \frac{2\theta}{\sqrt{2\pi}}.
\end{align*}
\label{lem:Insensitivity1}
\end{lemma}

\begin{proof}
Recall Equations~\eqref{eq:OrderedProb} and~\eqref{eq:relationLoadCap}, and observe that the case with $M_1=M_2=0$ is implied by Theorem~\ref{thm:CumulativeProbGeneral}. The general case follows from a straightforward calculation:
\begin{align*}
\sqrt{\frac{kN}{N-k}} &\Prob\left( U_{(i)}^{N-M_1} \leq \frac{\theta+i-1}{N-M_1},  \;\;\; i=1,...,k-M_2 \right) \\
 &=\underbrace{\sqrt{\frac{k}{k-M_2}}}_{\rightarrow 1} \underbrace{\sqrt{\frac{N}{N-M_1}}}_{\rightarrow 1} \underbrace{\sqrt{\frac{N-k+M_2-M_1}{N-k}}}_{\rightarrow 1}  \\
& \hspace{4cm} \cdot \underbrace{\sqrt{\frac{(k-M_2)(N-M_1)}{N-k+M_2-M_1}}  \cdot \Prob\left( U_{(i)}^{N-M_1} \leq \frac{\theta+i-1}{N-M_1}, \;\;\; i=1,...,k-M_2 \right)}_{\rightarrow \frac{2 \theta}{\sqrt{2\pi}}} .
\end{align*}
The last convergence follows from noting~\eqref{eq:OrderedProb} and applying Theorem~\ref{thm:CumulativeProbGeneral} to a network with $N-M_1$ lines and threshold $k-M_2$.
\end{proof}

\begin{lemma}
Let $k$ be a function of $N$ such that both $k \rightarrow \infty$ and $N-k \rightarrow \infty$ as $N \rightarrow \infty$. Then for every fixed $M \in \mathbb{N}$
\begin{align}
\lim_{N \rightarrow \infty}  \sqrt{\frac{kN}{N-k}} \Prob&\left( U_{(i)}^{N-M} \leq \frac{\theta+i-1}{N}, \;\;\; i=1,...,k \right)   = \frac{2\theta}{\sqrt{2\pi}}.
\label{eq:Insensitivity2}
\end{align}
\label{lem:Insensitivity2}
\end{lemma}
\begin{proof}
Note that 
\begin{align*}
\limsup_{N \rightarrow \infty}  \sqrt{\frac{kN}{N-k}} \Prob&\left( U_{(i)}^{N-M} \leq \frac{\theta+i-1}{N}, \;\;\; i=1,...,k \right) \\
& \hspace{3cm} \leq \limsup_{N \rightarrow \infty}  \sqrt{\frac{kN}{N-k}} \Prob\left( U_{(i)}^{N} \leq \frac{\theta+i-1}{N}, \;\;\; i=1,...,k \right) = \frac{2\theta}{\sqrt{2\pi}}.
\end{align*} 

To obtain a lower bound, we first consider the case $M=1$. Consider a Poisson process with unit rate where the epoch of the $i$'th event is denoted by $S_i=\sum_{j=1}^i E_j$ with $E_j$ standard independent exponential random variables for all $j \geq 1$. Note that, given $S_i=t$, the joint distribution of $(S_1,...,S_{i-1})$ is the same as the joint distribution of $i-1$ ordered independent uniform random variables on $(0,t)$. Therefore,~\eqref{eq:Insensitivity2} with $M=1$ is equivalent to
\begin{align*}
\lim_{N \rightarrow \infty} \sqrt{\frac{kN}{N-k}} \Prob\left( \frac{S_i}{S_{N}} \leq \frac{\theta+i-1}{N}, \;\;\; i=1,...,k \right) = \frac{2\theta}{\sqrt{2\pi}}.
\end{align*}
We observe that for every $\epsilon>0$,
\begin{align*}
\Prob\left( \frac{S_i}{S_{N+1}} \leq \frac{\theta+i-1}{N}, \;\;\; \forall i\leq k \right) &\leq \Prob\left( S_i \leq \frac{(\theta+i-1)S_{N+1}}{N},\;\;\; \forall i\leq k; E_{N+1} \leq \epsilon S_N \right)+ \Prob\left( E_{N+1} > \epsilon S_N \right)\\
&\leq \Prob\left( S_i \leq (\theta+\epsilon+i-1)\frac{S_{N}}{N}, \;\;\; \forall i\leq k \right) + \Prob\left( E_{N+1} > \epsilon S_N \right).
\end{align*}
Since
\begin{align*}
\Prob\left( E_{N+1} > \epsilon S_N \right) &= \E\left(e^{-\epsilon S_N}\right) = \E\left(e^{-\epsilon S_1}\right)^N  = \left(\frac{1}{1+\epsilon}\right)^N,
\end{align*}
it follows that
\begin{align*}
\liminf_{N \rightarrow \infty} &\sqrt{\frac{kN}{N-k}} \Prob\left( \frac{S_i}{S_{N}} \leq \frac{\theta+i-1}{N} , \;\;\; \forall i\leq k \right) \\
& \hspace{2cm}\geq \liminf_{N \rightarrow \infty} \sqrt{\frac{kN}{N-k}} \cdot \left(\Prob\left( \frac{S_i}{S_{N+1}} \leq \frac{\theta-\epsilon+i-1}{N}, \;\;\; \forall i\leq k \right) - \left(\frac{1}{1+\epsilon}\right)^N \right)= \frac{2(\theta-\epsilon)}{\sqrt{2\pi}},
\end{align*}
for every $\epsilon >0$. The result for $M=1$ follows by letting $\epsilon \downarrow 0$. Inductively applying the result for $M=1$ until the fixed $M >0$ is reached concludes the proof.
\end{proof}

In view of~\eqref{eq:1A},~\eqref{eq:OrderedProb} and~\eqref{eq:FinitePerturbLoadFunction}, it is convenient to introduce the stopping times
\begin{align}
\tau^N_{\theta,\Delta} = \min\left\{i \in \mathbb{N} : U^N_{(i)} > \frac{\theta+i-1+\Delta(i,N)}{N} \right\}-1
\label{eq:TauDefinition}
\end{align}
for all constants $\theta \in \mathbb{R}$ and functions $\Delta: \mathbb{N} \times \mathbb{N} \rightarrow \mathbb{R}$. In particular, if the constant $\theta$ and function $\Delta(\cdot,\cdot)$ are chosen as in~\eqref{eq:FinitePerturbLoadFunction}, then $A^N=\tau^N_{\theta,\Delta}$. Yet, the advantage of the notation as in~\eqref{eq:TauDefinition} appears when we compare the asymptotic exceedance probability for different constants $\theta$ and functions $\Delta(\cdot,\cdot)$. 

Our derivation of the asymptotic behavior of the exceedance probability makes use of similar arguments multiple times in the proof. We present these arguments separately by means of the next two lemmas.

\begin{lemma}
Let $k:=k(N) \leq N$ be a positive function of $N$ such that $k \rightarrow \infty$ and $N-k \rightarrow \infty$ as $N \rightarrow \infty$. Let $c_i$, $i \in \mathbb{N}$ be as in~\eqref{eq:DefinitionCAss} and for some fixed $M \in \mathbb{N}$, suppose $\Delta(i,N)=0$ for all $i \geq M$ and $N \geq N_0$ for some $N_0 \in \mathbb{N}$. For all constants $a,b \in \mathbb{R}_{\geq 0}$ with $a \leq b$,
\begin{align}
\begin{split}
\lim_{N \rightarrow \infty} \sqrt{\frac{kN}{N-k}} \Prob&\left(\tau^N_{\theta,\Delta} \geq k ; U_{(M)}^N \in \left[\frac{a}{N},\frac{b}{N}\right] \right)  \\
&\hspace{1.5cm}= \frac{2}{\sqrt{2 \pi}}\int_a^{b} \Prob\left(\Uimeen \leq \frac{c_{i}}{y}, \;\;\; \forall i \leq M-1 \right)  (\theta+M-y) \frac{y^{M-1}}{(M-1)!}e^{-y} \, dy .
\end{split}
\label{eq:InterchangeExpression}
\end{align}
\label{lem:ConditioningM}
\end{lemma}

\begin{proof}
The density of the $M$'th order statistic of a sample of $N$ standard uniformly distributed random variables is given by a beta distribution~\cite{Embrechts1997}
\begin{align*}
f_{\Umn}(x)=\frac{N!}{(M-1)!(N-M)!}x^{M-1}(1-x)^{N-M}.
\end{align*}
Conditioning on the $M$'th order statistic yields
\begin{align*}
\Prob(\tau^N_{\theta,\Delta} \geq k) &= \int_{\frac{a}{N}}^{\frac{b}{N}} \Prob\left(\Uin \leq \frac{c_{i,N}}{N}, \;\;\; \forall i \leq k \big| \Umn =x \right) \cdot f_{\Umn}(x) \, dx \\
&= \int_a^{b} \Prob\left(\Uin \leq \frac{c_{i,N}}{N}, \;\;\; \forall i \leq k \big| \Umn =\frac{y}{N} \right)  \cdot \frac{f_{\Umn}\left(\frac{y}{N}\right)}{N} \, dy\\
&= \int_a^{b} \Prob\left(\Uimeen \leq \frac{c_{i,N}}{y}, \;\;\; \forall i \leq M-1 \right) \\
&\hspace{4cm} \cdot \Prob\left(\Uinm \leq \frac{\theta+M-y+i-1}{N(1-\frac{y}{N})}, \;\;\; \forall i \leq k-M \right) \cdot \frac{f_{\Umn}\left(\frac{y}{N}\right)}{N}  \, dy.
\end{align*}
The latter equality follows from the Markov property: Given that $\Umn =y/N$, the first $M-1$ order statistics are independent of the other order statistics and distributed as $M-1$ uniformly distributed random variables on the interval $[0,y/N]$. Similarly, the other order statistics are independent of the first $M$ order statistics, and have the same law as $N-M$ uniformly distributed random variables on the interval $[y/N,1]$. Rescaling the intervals results in the above expression.

\noindent
Next, we show that an interchange of limit and integration is justified by bounding all three terms within the integral form above. First, we observe that for all $y \in [a,b]$,
\begin{align*}
\frac{f_{\Umn}\left(\frac{y}{N}\right)}{N} &= \frac{(N-1)!}{(M-1)!(N-M)!}\left(\frac{y}{N}\right)^{M-1}\left(1-\frac{y}{N}\right)^{N-M} \\
&\leq \frac{N^{M-1}}{(M-1)!}\left(\frac{y}{N}\right)^{M-1} \leq \frac{b^{M-1}}{(M-1)!} < \infty.
\end{align*}
Second, we show that the second term multiplied with $\sqrt{kN / (N-k)}$ is also bounded. Let $M^{\star}= \lceil b \rceil$, and hence for all $y \in [a,b]$, $N-M^{\star} \leq N-y \leq N$,
\begin{align*}
\Prob\left(\Uinm \leq \frac{\theta+M-y+i-1}{N(1-\frac{y}{N})}, \;\;\; \forall i \leq k-M \right) &\geq \Prob\left(\Uinm \leq \frac{\theta+M-y+i-1}{N}, \;\;\; \forall i \leq k-M \right), \\
\Prob\left(\Uinm \leq \frac{\theta+M-y+i-1}{N(1-\frac{y}{N})}, \;\;\; \forall i \leq k-M \right) &\leq \Prob\left(\Uinm \leq \frac{\theta+M-y+i-1}{N-M^{\star}}, \;\;\; \forall  i \leq k-M \right).
\end{align*}
Applying Lemmas~\ref{lem:Insensitivity1} and~\ref{lem:Insensitivity2} and subsequently the sandwich theorem yields
\begin{align*}
 \lim_{N \rightarrow \infty} \sqrt{\frac{kN}{N-k}}\Prob\left(\Uinm \leq \frac{\theta+M-y+i-1}{N(1-\frac{y}{N})} \right) &=\liminf_{N \rightarrow \infty} \Prob\left(\Uinm \leq \frac{\theta+M-y+i-1}{N}, \; \forall i \leq k-M \right) \\
&=\limsup_{N \rightarrow \infty} \Prob\left(\Uinm \leq \frac{\theta+M-y+i-1}{N-M^{\star}}, \; \forall  i \leq k-M \right) \\
&= \frac{2(\theta+M-y)}{\sqrt{2 \pi}} .
\end{align*}

\noindent
We find that the second term multiplied with $\sqrt{kN / (N-k)}$ is indeed bounded, since every converging sequence is bounded. Finally, the first term is trivially bounded by one, and therefore the dominated convergence theorem justifies an interchange of limit and integration. Since $U_{(i)}^{M-1}$, $i=1,...,M-1$ have a density not depending on $N$, it holds that
\begin{align*}
\lim_{N \rightarrow \infty}&\Prob\left( \Uimeen \leq \frac{c_{i,N}}{y}, \;\;\; \forall  i \leq M-1 \right)  = \Prob\left( \Uimeen \leq \frac{c_{i}}{y}, \;\;\; \forall  i \leq M-1 \right),
\end{align*}
and moreover,
\begin{align*}
&\lim_{N \rightarrow \infty} \frac{1}{N} \, f_{\Umn}\left(\frac{y}{N}\right) = \frac{y^{M-1}}{(M-1)!}e^{-y}.
\end{align*}
We conclude that~\eqref{eq:InterchangeExpression} holds.
\end{proof}

In order to obtain a more quantitative handle on the integral expression in~\eqref{eq:InterchangeExpression}, we need to have a deeper understanding of the probability term within the integral. The next lemma expresses this probability by means of a recursive formula, and we refer the reader to the appendix for the proof.

\begin{lemma}
Let $M \in \mathbb{N}$ be fixed, and suppose $\Delta(i,N)=0$ for all $i \geq M$ and $N \geq N_0$ for some $N_0 \in \mathbb{N}$. Let $c_i$, $i \in \mathbb{N}$ be as in~\eqref{eq:DefinitionCAss} and for every $y \in \mathbb{R}_{\geq 0}$, define $\sigma_M(y)=0$ if $c_1>y$ and otherwise
\begin{align*}
\sigma_M(y) = \max\{ i \in \mathbb{N} : i \leq M, c_i < y \}.
\end{align*}
Then,
\begin{align}
\Prob &\left( U_{(i)}^M \leq \frac{c_i}{y}, \;\;\; \forall i \leq M \right)  = 1-\frac{M!}{y^M} \sum_{j=1}^{\sigma_M(y)} \beta_{j-1} \frac{(y-c_j)^{M-j+1}}{(M-j+1)!},
\end{align}
where $\beta_0=1$ and for all $i \geq 1$,
\begin{align}
\beta_i = \sum_{j=1}^i \frac{(-1)^{j+1}}{j!} \beta_{i-j} (c_{i-j+1})^j.
\label{eq:BetaDefinition}
\end{align}
\label{lem:QuantifyProb}
\end{lemma}
 
The previous lemmas provide the building blocks to prove Theorem~\ref{thm:PerturbConstant}. As a first step, we consider a scenario with only finitely many perturbations, see Figure~\ref{fig:FluctuationIllustration}.

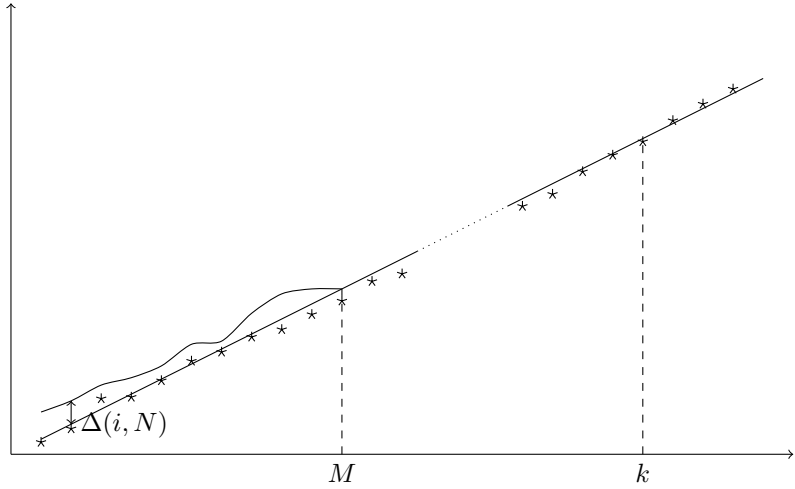
\begin{figure}[htb]
\centering
\begin{tikzpicture}[xscale=0.4, yscale=0.2]
\draw[->] (0,0) -- (0,30);
\draw[->] (0,0) -- (26,0);

\draw plot[smooth] file {fluct.txt};
\draw (1,1) -- (13.5,13.5);
\draw[dotted] (13.5,13.5) -- (16.5,16.5);
\draw (16.5,16.5) -- (25,25);
\draw [<->] (2,2) -- (2,3.551739230573947);
\draw node[right] at (2,2) {$\Delta(i,N)$};

\draw [dashed] (11,0) -- (11,11);
\draw node[below] at (11,0) {$M$};

\draw [dashed] (21,0) -- (21,21);
\draw node[below] at (21,0) {$k$};
\draw plot[only marks, mark=star, scale plot marks=false, mark size=2] coordinates {(1,0.8) (2,1.7) (3,3.7) (4,3.8) (5,4.9) (6,6.2) (7,6.8) (8,7.8) (9,8.3) (10,9.3) (11,10.2) (12,11.5) (13,12) (17,16.5) (18,17.3) (19,18.8) (20,19.9) (21,20.8) (22, 22.2) (23, 23.3) (24, 24.3)};
\end{tikzpicture}
\caption{Effect of perturbations for the truncated case.}
\label{fig:FluctuationIllustration}
\end{figure}

\begin{proposition}
Let $k:=k(N) \leq N$ be a positive function of $N$ such that $k \rightarrow \infty$ and $N-k \rightarrow \infty$ as $N \rightarrow \infty$. Let $F\circ l^N$ be as in~\eqref{eq:FinitePerturbLoadFunction} with $\Delta(i,N)=0$ for all $i \geq M$ and $N \geq N_0$ for some fixed $M \in \mathbb{N}$ and $N_0 \in \mathbb{N}$, and let $c_i, i \in \mathbb{N}$ be as in~\eqref{eq:DefinitionCAss}. Then, there exists a constant $V_M(\theta,\Delta) \in (0,\infty)$ such that
\begin{align*}
&\lim_{N \rightarrow \infty} \sqrt{\frac{kN}{N-k}}\Prob(\A \geq k) = V_M(\theta,\Delta).
\end{align*}
Let $\beta_i$, $ i \in \mathbb{N}$, be as in~\eqref{eq:BetaDefinition} and let $\gamma(\cdot,\cdot)$ denote the lower incomplete gamma distribution:
\begin{align*}
\gamma(s,x)=\int_0^x t^{s-1} e^{-t} \, dt.
\end{align*}
The value of $V_M(\theta,\Delta)$ can be expressed as
\begin{align*}
V_M(\theta,\Delta) &= \frac{2}{\sqrt{2 \pi}}\left( \frac{\theta }{(M-1)!}\gamma(M, c_M)+\frac{(c_M)^M}{(M-1)!} e^{-c_M} \right. \\
& \hspace{2cm}\left.+\sum_{j=1}^{M-1} \beta_{j-1} e^{-c_j} \frac{\Delta(j)}{(M-j)!} \gamma(M-j+1,c_M-c_j)   -\sum_{j=1}^{M-1} \beta_{j-1}  \frac{(c_M-c_j)^{M-j+1}}{(M-j)!} e^{-c_M} \right).
\end{align*}
\label{prop:TruncatedExceedanceProb}
\end{proposition}

\begin{proof}
Noting~\eqref{eq:OrderedProb} and~\eqref{eq:TauDefinition}, applying Lemma~\ref{lem:ConditioningM} with $a=0$ and $b=c_M$ and subsequently invoking Lemma~\ref{lem:QuantifyProb} yields
\begin{align*}
\lim_{N \rightarrow \infty} \sqrt{\frac{kN}{N-k}}\Prob(\A \geq k) &= \frac{2}{\sqrt{2 \pi}} \int_0^{c_M} \Prob\left(\Uimeen \leq \frac{c_{i}}{y}, \;\;\; \forall i \leq M-1 \right)(\theta+M-y) \frac{y^{M-1}}{(M-1)!}e^{-y} \, dy\\
&= \frac{2}{\sqrt{2 \pi}}\int_0^{c_M} (\theta+M-y)\frac{y^{M-1}}{(M-1)!} e^{-y} \, dy \\
&\hspace{3cm}-\frac{2}{\sqrt{2 \pi}} \int_0^{c_M}  \theta+M-y) \sum_{j=1}^{\sigma_{M-1}(y)}  \beta_{j-1} \frac{(y-c_j)^{M-j}}{(M-j)!}e^{-y} \, dy.
\end{align*}
The first term can also be expressed as
\begin{align*}
\int_0^{c_M}  (\theta+M-y)\frac{y^{M-1}}{(M-1)!} e^{-y} \, dy &= \frac{(\theta+M)\gamma(M,c_M)}{(M-1)!}  - \frac{M \gamma(M,c_M)}{(M-1)!} + \frac{(c_M)^M }{(M-1)!} e^{-c_M} \\
&= \frac{\theta }{(M-1)!}\gamma(M, c_M)+\frac{(c_M)^M}{(M-1)!} e^{-c_M} .
\end{align*}

\noindent
The second term yields 
\begin{align*}
\int_0^{c_M}  (\theta&+M-y) \sum_{j=1}^{\sigma_{M-1}(y)}  \beta_{j-1} \frac{(y-c_j)^{M-j}}{(M-j)!}e^{-y} \, dy =\sum_{m=1}^{M-1}  \sum_{j=1}^m \int_{c_m}^{c_{m+1}} (\theta+M-y)  \beta_{j-1} \frac{(y-c_j)^{M-j}}{(M-j)!}e^{-y} \, dy \\
&= \sum_{j=1}^{M-1} \int_{c_j}^{c_{M}} \beta_{j-1} (\theta+M-y) \frac{(y-c_j)^{M-j}}{(M-j)!}e^{-y} \, dy \\
&= \sum_{j=1}^{M-1} \beta_{j-1} e^{-c_j} \int_{0}^{c_{M}-c_j} \frac{(\theta+M-c_j-u)u^{M-j}}{(M-j)!}e^{-u} \, du \\
&= \sum_{j=1}^{M-1} \beta_{j-1} e^{-c_j} \left( \frac{\theta+M-c_j}{(M-j)!} \gamma(M-j+1,c_M-c_j) \right.\\
&\hspace{4.5cm} - \frac{M-j+1}{(M-j)!} \gamma(M-j+1,c_M-c_j) \left. + \frac{(c_M-c_j)^{M+j-1}}{(M-j)!} e^{-(c_M-c_j)}\right) \\
&= \sum_{j=1}^{M-1} \beta_{j-1} e^{-c_j} \frac{-\Delta(j)}{(M-j)!} \gamma(M-j+1,c_M-c_j) +\sum_{j=1}^{M-1} \beta_{j-1} \frac{(c_M-c_j)^{M+j-1}}{(M-j)!} e^{-c_M}.
\end{align*}
Subtracting the second term from the first concludes the proof.
\end{proof}

Next, we allow for all perturbations that satisfy conditions~(A)-(D). It turns out that it is more convenient to use an equivalent condition of~(C) and~(D): for every $\epsilon >0$ there exists a pair $(M_\epsilon,N_\epsilon) \in \mathbb{N} \times \mathbb{N}$ such that $|\Delta(i,N)|< \epsilon$ for all $N \geq N_\epsilon$ and all $M_\epsilon \leq i \leq k(N)$. We refer the reader to Lemma~\ref{lem:EquivalentCondition} in the Appendix for a formal proof that the conditions are equivalent.
\begin{figure}[htb]
\centering
\begin{tikzpicture}[xscale=3, yscale=6]
\draw[->,thick] (0,0) -- (0,1.2);
\draw[->, thick] (0,0) -- (4.2,0);

\draw[thick, color=brown] plot[smooth] file {ThetaLine.txt};
\draw plot[smooth] file {InfinitePerturb2.txt};

\draw[dotted] (0,0.25) -- (4,1.045);
\draw[dotted] (0,0.15) -- (4,0.955);

\draw[thick, dashed, color=red] plot[smooth] file {UBReal2.txt};
\draw[thick, dashed, color=red] (1.8,0.605) -- (4,1.045);
\draw[thick, dashed, color=blue] plot[smooth] file {UBReal2.txt};
\draw[thick, dashed, color=blue] (1.8,0.515) -- (1.8,0.605);
\draw[thick, dashed, color=blue] (1.8,0.515) -- (4,0.955);

\draw[thick, dashed] (1.8,0) -- (1.8,0.6);
\draw node[below] at (1.8,0) {$M_\epsilon$};
\draw[<->, thick] (3.9,0.98) -- (3.9,1.025);
\draw node[above] at (3.9,1.02) {$\epsilon$};

\draw[<->] (0.30,0.26) -- (0.30,0.17);
\draw node[below right] at (0.30,0.17) {$\Delta$};

\end{tikzpicture}
\caption{Illustration of infinitely many perturbations setting.}
\label{fig:InfiniteFluctuationIllustration}
\end{figure}
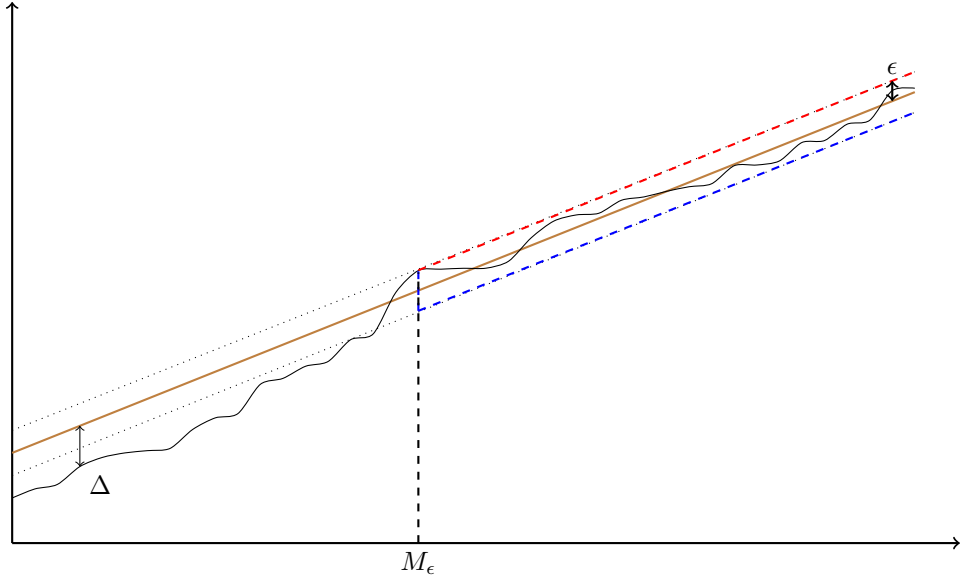

We show that the exceedance probability times $\sqrt{kN/(N-k)}$ still converge to a constant by considering the bounds illustrated by the dashed lines in Figure~\ref{fig:InfiniteFluctuationIllustration} for every fixed $\epsilon>0$. That is, for an upper bound, we consider the exceedance probability in case of an initial disturbance $(\theta+\epsilon)/N$ and allowing for the first $M_\epsilon -1$ perturbations. Indeed, this yields an upper bound for all $N \geq N_\epsilon$: the values are the same for all pairs $(i,N)$ with $i \leq M_\epsilon -1$, and for $i \geq M_\epsilon$, we have $\theta+i-1+\Delta(i,N) \leq \theta+\epsilon+i-1$ for all $N \geq N_\epsilon$. Similarly, for a lower bound we consider the case with initial disturbance $(\theta-\epsilon)/N$ where we allow for the first $M_\epsilon -1$ perturbations. By applying Proposition~\ref{prop:TruncatedExceedanceProb}, we can determine the asymptotic behavior of the bounds explicitly. We show that as $\epsilon \downarrow 0$, the upper and lower bound converges to the same constant $V(\theta,\Delta) \in (0,\infty)$ defined as
\begin{align}
V(\theta,\Delta) = \lim_{M \rightarrow \infty} V_M(\theta,\Delta).
\label{eq:VconstantDefinition}
\end{align}

\begin{proof}[Proof of Theorem~\ref{thm:PerturbConstant}]
By assumption and Lemma~\ref{lem:EquivalentCondition}, we know that $\forall \epsilon >0$ there exists a pair $(M_\epsilon ,N_\epsilon) \in \mathbb{N} \times \mathbb{N}$ such that $|\Delta(i,N)| < \epsilon$ for every $N \geq N_\epsilon$ and $M_\epsilon \leq i \leq k(N)$. Fix $\epsilon>0$, and define for all $(i,N) \in \mathbb{N} \times \mathbb{N}$,
\begin{align*}
\Delta_1(i,N) =  \left\{ \begin{array}{ll}
\Delta(i,N)-\epsilon & \textrm{if } i<M_\epsilon, \\
0 & \textrm{if } i\geq M_\epsilon,
\end{array} \right. 
\end{align*}
and
\begin{align*}
\Delta_2(i,N) =  \left\{ \begin{array}{ll}
\Delta(i,N)+\epsilon & \textrm{if } i<M_\epsilon, \\
0 & \textrm{if } i\geq M_\epsilon.
\end{array} \right.
\end{align*} 
Recall the definition of the stopping times defined in~\eqref{eq:TauDefinition} and particularly, $\tau^N_{\theta,\Delta}=A^N$. Observe that the case of the upper and lower bound described above thus correspond to stopping times $\tau^N_{\theta+\epsilon,\Delta_1}$ and $\tau^N_{\theta-\epsilon,\Delta_2}$ respectively. Applying Proposition~\ref{prop:TruncatedExceedanceProb} to these cases with $M=M_\epsilon$
yields
\begin{align*}
\lim_{N \rightarrow \infty} \sqrt{\frac{kN}{N-k}} \Prob(\tau^N_{\theta+\epsilon,\Delta_1} \geq k) = V_{M_\epsilon}(\theta+\epsilon,\Delta_1), \\
\lim_{N \rightarrow \infty} \sqrt{\frac{kN}{N-k}} \Prob(\tau^N_{\theta-\epsilon,\Delta_2} \geq k) = V_{M_\epsilon}(\theta-\epsilon,\Delta_2).
\end{align*}
Couple $\tau^N_{\theta+\epsilon,\Delta_1}$, $\tau^N_{\theta,\Delta}=A^N$ and $\tau^N_{\theta-\epsilon,\Delta_2}$. Then the inequalities $\tau^N_{\theta-\epsilon,\Delta_2} \leq A^N \leq \tau^N_{\theta+\epsilon,\Delta_1}$ hold, and hence we obtain
\begin{align*}
\lim_{N \rightarrow \infty} \sqrt{\frac{kN}{N-k}}\Prob&(A^N \geq k)\in \left[V_{M_\epsilon}(\theta-\epsilon,\Delta_2), V_{M_\epsilon}(\theta+\epsilon,\Delta_1)\right].
\end{align*} 

Next, we show that the limits of $V_{M_\epsilon}(\theta-\epsilon,\Delta_2)$ and $V_{M_\epsilon}(\theta+\epsilon,\Delta_1)$ coincide as $\epsilon \downarrow 0$, i.e.
\begin{align*}
&\lim_{\epsilon \downarrow 0} \left[V_{M_\epsilon}(\theta+\epsilon,\Delta_1) -  V_{M_\epsilon}(\theta-\epsilon,\Delta_2) \right] =0
\end{align*}
For this, we condition on the value of $U_{(M_\epsilon)}^N$:
\begin{align*}
V_{M_\epsilon}(\theta+\epsilon,\Delta_1) - V_{M_\epsilon}(\theta-\epsilon,\Delta_2) &=\lim_{N \rightarrow \infty} \sqrt{\frac{kN}{N-k}} \left(\Prob(\tau^N_{\theta+\epsilon,\Delta_1} \geq k) - \Prob(\tau^N_{\theta-\epsilon,\Delta_2} \geq k) \right) \\
&\leq
z_1(\epsilon)+z_2(\epsilon),
\end{align*}
where
\begin{align*}
z_1&(\epsilon) =  \lim_{N \rightarrow \infty} \sqrt{\frac{kN}{N-k}}  \left(\Prob \left(\tau^N_{\theta+\epsilon,\Delta_1} \geq k ; U_{(M_\epsilon)}^N \in I_1 \right) - \Prob\left(\tau^N_{\theta-\epsilon,\Delta_2} \geq k ; U_{(M_\epsilon)}^N \in I_1 \right) \right)
\end{align*}
and 
\begin{align*}
z_2(\epsilon) =  \lim_{N \rightarrow \infty} \sqrt{\frac{kN}{N-k}}  \Prob \left(\tau^N_{\theta+\epsilon,\Delta_1} \geq k ; U_{(M_\epsilon)}^N \in I_2 \right),
\end{align*}
with $I_1 = \left[ 0 , \frac{\theta-\epsilon+M_\epsilon-1}{N} \right]$ and $I_2=\left[ \frac{\theta-\epsilon+M_\epsilon-1}{N}, \frac{\theta+\epsilon+M_\epsilon-1}{N} \right]$.

Note that for all $i < M_\epsilon$ and $N \in \mathbb{N}$, $c_{i,N}$ are the same for $\tau^N_{\theta+\epsilon,\Delta_1}$ and $\tau^N_{\theta+\epsilon,\Delta_1}$ by definition of $\Delta_1$ and~$\Delta_2$. Applying Lemma~\ref{lem:ConditioningM} to $z_1(\epsilon)$, we obtain
\begin{align*}
z_1(\epsilon) &= \frac{2}{\sqrt{2 \pi}}\int_0^{\theta-\epsilon+M_\epsilon-1} \Prob\left(U_{(i)}^{M_\epsilon-1} \leq \frac{c_i}{y}, \;\;\;\forall i \leq M_\epsilon-1 \right) \\
&\hspace{4cm}\cdot \left( (\theta+\epsilon+M_\epsilon-y) \frac{y^{M_\epsilon-1}}{(M_\epsilon-1)!} e^{-y} -(\theta-\epsilon+M_\epsilon-y)\frac{y^{M_\epsilon-1}}{(M_\epsilon-1)!} e^{-y} \right) \, dy \\
&\leq \frac{4 \epsilon}{\sqrt{2\pi}}\int_0^{\theta+\epsilon+M_\epsilon-1} \frac{y^{M_\epsilon-1}}{(M_\epsilon-1)!} e^{-y} \, dy \leq \frac{4 \epsilon}{\sqrt{2\pi}}.
\end{align*}
For every fixed $M \in \mathbb{N}$, differentiating $y^M e^{-y}$ with respect to $y$ and determining its roots shows that this function has one maximum attained at $y=M$ and hence, $y^M e^{-y} \leq M^M e^{-M}$. Using Lemma~\ref{lem:ConditioningM}, the previous argument and Stirling's bound yields
\begin{align*}
z_2(\epsilon) &=  \frac{2}{\sqrt{2 \pi}}\int_{\theta-\epsilon+M_\epsilon-1}^{\theta+\epsilon+M_\epsilon-1} \Prob\left(U_{(i)}^{M-1} \leq \frac{c_i}{y}, \;\;\;\forall i \leq M-1 \right) (\theta+\epsilon+M_\epsilon-y) \frac{y^{M-1}}{(M-1)!} e^{-y} \, dy \\
&\leq \frac{2}{\sqrt{2 \pi}} \int_{\theta-\epsilon+M_\epsilon-1}^{\theta+\epsilon+M_\epsilon-1} (1+2\epsilon) \frac{(M-1)^{M-1}}{(M-1)!} e^{-(M-1)} \, dy \leq \frac{4 \epsilon(1+2\epsilon)}{\sqrt{2\pi}}.
\end{align*}
Consequently,
\begin{align*}
V_{M_\epsilon}(\theta+\epsilon,\Delta_1)- V_{M_\epsilon}&(\theta-\epsilon,\Delta_2) \leq \frac{8\epsilon(1+\epsilon)}{\sqrt{2 \pi}},
\end{align*}
and since the difference is positive, it must converge to zero as $\epsilon \downarrow 0$.

What remains to be shown is that the limit of $V_{M_\epsilon}(\theta+\epsilon,\Delta_1)$ exists, and thus also $V_{M_\epsilon}(\theta-\epsilon,\Delta_2)$, and is the same as $V(\theta,\Delta)$ defined in~\eqref{eq:VconstantDefinition}. The existence of the limit follows from the monotonicity of $V_{M_\epsilon}(\theta+\epsilon,\Delta_1)$. That is, $V_{M_\epsilon}(\theta+\epsilon,\Delta_1)$ is non-decreasing and bounded from below by a strictly positive constant, for example $V_{M_\epsilon}(\theta-\epsilon,\Delta_2)$ with $\epsilon=1$. Since every monotone bounded function in a complete metric space converges, it follows that the limit exists as $\epsilon \downarrow 0$. Moreover, since $V(\theta,\Delta) \in [V(\theta-\epsilon,\Delta_2),V(\theta+\epsilon,\Delta_1)]$ for every $\epsilon >0$, the value of the limit must in fact be $V(\theta,\Delta)$.
\end{proof}

Suppose that for a fixed $\epsilon>0$ we determined the pair $(M_\epsilon,N_\epsilon)$ such that $|\Delta(i,N)|< \epsilon$ for all $N \geq N_\epsilon$ and all $M_\epsilon \leq i \leq k(N)$. Since $V(\theta,\Delta)$ lies between $V(\theta-\epsilon,\Delta_2)$ and $V(\theta+\epsilon,\Delta_1)$, it follows from the proof of Theorem~\ref{thm:PerturbConstant} that 
\begin{align*}
\left|V(\theta,\Delta)-V_{M_\epsilon}(\theta,\Delta)\right| \leq \frac{8\epsilon(1+\epsilon)}{\sqrt{2 \pi}}.
\end{align*}
This observation gives rise to Algorithm~\ref{alg:1}.

\begin{center}
\begin{algorithm}
\KwIn{Target error $\delta>0$, constant $\theta$ and perturbations $\Delta(\cdot,\cdot)$.}
\KwOut{Approximation $V_{M_\epsilon}(\theta,\Delta)$ such that $\left|V(\theta,\Delta)-V_{M_\epsilon}(\theta,\Delta)\right|<\delta$.}
\begin{enumerate}
\item Determine $\epsilon >0$ such that $\frac{8\epsilon(1+\epsilon)}{\sqrt{2 \pi}} \leq \delta$.
\item Determine pair $(M_\epsilon,N_\epsilon)$ such that $|\Delta(i,N)|< \epsilon$ \\ for all $N \geq N_\epsilon$ and all $M_\epsilon \leq i \leq k(N)$.
\item Return $V_{M_\epsilon}(\theta,\Delta)$ defined in Proposition~\ref{prop:TruncatedExceedanceProb}.
\end{enumerate}
\caption{Approximation scheme for $V(\theta,\Delta)$.}
\label{alg:1}
\end{algorithm}
\end{center}

\section{Identifying when power-law behavior prevails}
\label{sec:IdentifyCriticalRegion}
When the surplus capacity distribution and/or load surge function are given, we would like to know which (growing) thresholds $k:=k(N)$, if any, yield power-law behavior for the exceedance probability. A sufficient condition is provided by Theorem~\ref{thm:PerturbConstant} and accordingly, we need to determine the thresholds $k$ such that (C) and (D) are satisfied. Such thresholds can be derived by exploiting the Taylor expansion
\begin{align}
F(\LoadHat) &= F'(0)\LoadHat + O((\LoadHat)^2).
\label{eq:TaylorApproximation}
\end{align}
We observe that Equation~\eqref{eq:TaylorApproximation} leads to an approximation of the composition that only requires information on the value~$F'(0)$ and the load surge function. That is, the only property of the surplus capacity distribution we need for checking whether power-law behavior prevails, is its behavior near its minimum. In particular, the average of the surplus capacity does not play any role.

We close this section by setting the threshold $k$ to a certain fixed integer, which allows us to analyze cases where the perturbations do not satisfy conditions~(C) and~(D). We suggest a method to explore the asymptotic behavior numerically for these cases.

\begin{example}\normalfont 
The main purpose of this example is to illustrate how we can use the Taylor expansion of the composition $F\circ l^N$ to determine growing thresholds $k$ where power-law behavior for the exceedance probability prevails. Suppose we have a surplus capacity distribution with density $F'(0)$ in zero, and let the load surge function be given by
\begin{align*}
l^N(i) = \frac{\theta+i-1}{N F'(0)}
\end{align*}
for some constant $\theta>0$. If the surplus capacity is uniformly distributed, we  have the setting of~\cite{Dobson2005} and all thresholds $k$ such that both $k \rightarrow \infty$ and $N-k \rightarrow \infty$ as $N\rightarrow \infty$ lead to power-law behavior for the exceedance probability. Otherwise, due to~\eqref{eq:TaylorApproximation}, the perturbations are given by
\begin{align*}
\Delta(i,N) &= O\left( \frac{i^2}{N} \right)
\end{align*}
for all $(i,N) \in \mathbb{N} \times \mathbb{N}$. Therefore, $\Delta(i) = \lim_{N \rightarrow \infty} \Delta(i,N) = 0$ for all $i \in \mathbb{N}$ and condition (C) is satisfied. For (D) to hold, we need $k=o(\sqrt{N})$, because for such thresholds $\Delta(i,N) \rightarrow 0$ for all $i \leq k$ as $N \rightarrow \infty$.
\end{example}

\begin{example}\normalfont 
\label{ex:Shortle}
Next, we verify and formalize the claims for the model in~\cite{Shortle2013} that we discussed in Section~\ref{sec:MainResults}. Recall that the load surge function is given by
\begin{align*}
l^N(i) = \frac{aN}{N-i} - a = \frac{a i}{N-i},
\end{align*} 
and suppose that $a = 1/F'(0)$. Then, applying the Taylor expansion~\eqref{eq:TaylorApproximation}, we obtain
\begin{align*}
\Delta(i,N) &=  O\left( \left(\frac{i}{N }\right)^2 \right) +O\left(N \left(\frac{i}{N-i}\right)^2\right)
\end{align*}
for all $(i,N) \in \mathbb{N} \times \mathbb{N}$. Again, we have the pointwise convergence $\Delta(i)=0$ for all $i \in \mathbb{N}$. In addition, we require that $k=o(\sqrt{N})$ for condition~(D) to hold for all $i \leq k$.

We emphasize that~\eqref{eq:TaylorApproximation} yields very rough bounds, and when more specific information is known about the surplus capacity distribution, more sophisticated bounds can lead to larger possible thresholds $k$. For instance, if $a=1$ and the surplus capacity is uniformly distributed in the previous example, then $\Delta(i,N)= O((i/N)^2)$ and power-law behavior for the exceedance probability occurs for all $k =o(N)$.
\end{example}

\begin{example}\normalfont 
In this example we suggest a numerical method for exploring the asymptotic  behavior of the exceedance probability for setting where the perturbations do not necessarily satisfy conditions~(C) and~(D). Intuitively, we find that if the value $\Delta(1,N)$ tends too close to its lower bound as $N \rightarrow \infty$, the system does not perceive an initial disturbance and no lines will fail. On the other hand, if $\Delta(k,N)$ becomes too large as $N \rightarrow \infty$, the system cannot deal with such a strong increase of load and the threshold $k$ will certainly be exceeded. If $\Delta(1,N)$ is not too small and $\Delta(k,N)$ is not too large as $N \rightarrow \infty$, we obtain a non-degenerate limit for the exceedance probability.

\begin{proposition}
Let $c_{i,N} := N \cdot F ( l^N(i))$ for $(i,N) \in \mathbb{N} \times \mathbb{N}$ and $c_{i} = \lim_{N \rightarrow \infty} c_{i,N}$ for $i \in \mathbb{N}$, which is a non-decreasing sequence. If $c_{1}>0$ and $c_{k} = O(1)$, then
\begin{equation}
\lim_{N \rightarrow \infty} \Prob(\A \geq k) = 1-\sum_{j=1}^k \beta_{j-1} e^{-j}
\label{eq:TheoremFailedLined}
\end{equation}
with $\beta_i$, $i \in \mathbb{N}$, are defined as in~\eqref{eq:BetaDefinition} in Section~\ref{sec:ProofsOfC}.
\label{prop:AssFailedLines}
\end{proposition}

This result can be proven by applying results from extreme value theory, see in the Appendix for a formal proof. Proposition~\ref{prop:AssFailedLines} thus provides a method to determine the asymptotic exceedance probability for every fixed~$k$. For scenarios that do not satisfy the criteria we assumed in this paper, one can solve for the asymptotic exceedance probability numerically for large values of $k$ and explore its behavior in other regimes as well.
\end{example}

\section{Summary and outlook}
\label{sec:Outlook}
The model of Dobson et al.~\cite{Dobson2005} shows power-law dependence for the exceedance probability when the system is critically loaded. In this paper, we identify settings where the power law prevails by extending the setting of~\cite{Dobson2005} in two directions. First, we show that the threshold can grow with the network size. Second, we consider broader load surge functions and surplus capacity distributions. We show that the power-law distribution prevails when the composition of the surplus capacity distribution function and the load surge function ultimately tends to a linearly increasing function with critical slope.

However, for general load surge functions and surplus capacity distributions the power-law behavior will not continue to hold for all network size dependent thresholds. It would be of interest to determine all settings where the exceedance probability follows a power-law distribution, and to identify the behavior of the blackout size beyond these settings. We intend to pursue this in future research.

\begin{acknowledgement} \normalfont
This work is financially supported by the Netherlands Organization for Scientific Research (NWO) through the Gravitation NETWORKS grant 024.002.003, and by an NWO VICI grant.
\end{acknowledgement}

% or


\bibliographystyle{plain}
\bibliography{bibSummary}

\newpage
\section*{Appendix}
\label{sec:App}
\begin{lemma}
For the $\beta_k$, $k\geq1$ defined as in~\eqref{eq:BetaDefinition},
\begin{align*}
\beta_k = \int_{-c_{1}}^0 \int_{-c_{2}}^{y_1} \cdots \int_{-c_{k}}^{y_{k-1}} dy_k \cdots dy_1.
\end{align*}
\label{lem:betaIdentity}
\end{lemma}

\noindent
\begin{proof}
The proof is by induction. For $k=1$ we indeed have $\int_{-c_{1}}^0 dy_1=c_{1}=\beta_1$. Suppose the lemma holds for all integers strictly smaller than $k$. Then,
\begin{align*}
\int_{-c_{1}}^0 \int_{-c_{2}}^{y_1} \cdots \int_{-c_{k}}^{y_{k-1}} \; dy_k \cdots dy_1 &= \int_{-c_{1}}^0 \int_{-c_{2}}^{y_1} \cdots \int_{-c_{k-1}}^{y_{k-2}} y_{k-1} \; dy_{k-1} \cdots dy_1 +c_{k} \beta_{k-1}  \\
&= \int_{-c_{1}}^0 \int_{-c_{2}}^{y_1} \cdots \int_{-c_{k-2}}^{y_{k-3}} \frac{(y_{k-2})^2}{2} \; dy_{k-2} \cdots dy_1 - \frac{(c_{k-1})^2}{2}\beta_{k-2} +c_{k} \beta_{k-1}  \\
&= \int_{-c_{1}}^0 \frac{1}{(k-1)!} y_1^{k-1} \; dy_1 + \sum_{j=1}^{k-1} \frac{(-1)^{j+1}}{j!} \beta_{k-j}(c_{k-j+1})^j \\
&= \sum_{j=1}^{k} \frac{(-1)^{j+1}}{j!} \beta_{k-j}(c_{k-j+1})^j=\beta_k.
\end{align*}
\end{proof}

\begin{lemma}
Let $(c_i)_{i \in \mathbb{N}}$ be a non-negative, non-decreasing sequence and $x \geq c_k$. For every $k \in \mathbb{N}$, it holds that
\begin{align*}
\beta_k + \sum_{j=1}^k \beta_{j-1} \sum_{l=0}^{k+1-j} \frac{(x-c_j)^l}{l!} = \sum_{l=0}^k \frac{x^k}{k!}.
\end{align*}
\label{lem:BetaIdentity2}
\end{lemma}

\begin{proof}
Particularly, we note that the identity is true for $k=0$. For $k \geq 1$, we note that due to the binomial formula, we have
\begin{align*}
\sum_{j=1}^k  \beta_{j-1} \sum_{l=0}^{k+1-j} \frac{(x-c_j)^l}{l!} &= \sum_{j=1}^k \sum_{l=0}^{k+1-j} \sum_{m=0}^{l} \beta_{j-1} \binom{l}{m} \frac{x^m (-c_j)^{l-m}}{l!} = \sum_{j=1}^k \sum_{m=0}^{k+1-j} \sum_{l=m}^{k+1-j} \beta_{j-1} \frac{x^m (-c_j)^{l-m}}{m! (l-m)!} \\
&= \sum_{j=1}^{k}  \sum_{l=0}^{k+1-j} \beta_{j-1} \frac{ (-c_j)^{l}}{l!} + \sum_{m=1}^{k} \frac{x^m}{m!}\sum_{j=1}^{k+1-m}  \sum_{l=m}^{k+1-j} \beta_{j-1} \frac{(-c_j)^{l-m}}{(l-m)!} \\
&= \sum_{j=1}^{k}  \beta_{j-1} +  \sum_{j=1}^{k}  \sum_{l=1}^{k+1-j} \beta_{j-1} \frac{ (-c_j)^{l}}{l!} +  \sum_{m=1}^{k} \frac{x^m}{m!}\sum_{j=1}^{k+1-m}  \beta_{j-1} \\
& \hspace{6cm} + \sum_{m=1}^{k} \frac{x^m}{m!}\sum_{j=1}^{k-m}  \sum_{l=1}^{k+1-m-j} \beta_{j-1} \frac{(-c_j)^{l}}{l!}.
\end{align*}
In the third term and the final term, we observe a double summation for all pairs of integers in a triangle. We apply the variable substitution $u=l+j-1$ and $v=l$ to sum over all pairs in the triangle via the diagonal lines. For the third term this yields
\begin{align*}
\sum_{j=1}^{k}  \sum_{l=1}^{k+1-j} \beta_{j-1} \frac{ (-c_j)^{l}}{l!} &= \sum_{u=1}^{k}  \sum_{v=1}^{u} \beta_{u-v} \frac{ (-c_{u-v+1})^{v}}{v!} = - \sum_{u=1}^{k} \beta_u.
\end{align*}
Applying the same argument to the fifth term yields
\begin{align*}
\sum_{j=1}^k \beta_{j-1} \sum_{l=0}^{k+1-j} \frac{(x-c_j)^l}{l!} &= \sum_{j=1}^{k}  \beta_{j-1}  - \sum_{u=1}^{k} \beta_u +  \sum_{m=1}^{k} \frac{x^m}{m!}  \left(  \sum_{j=1}^{k+1-m}  \beta_{j-1} - \sum_{u=1}^{k-m} \beta_u \right) \\
&= \beta_0-\beta_k + \beta_0 \sum_{m=1}^{k} \frac{x^m}{m!} = -\beta_k + \sum_{m=0}^{k} \frac{x^m}{m!},
\end{align*}
which proves the identity.
\end{proof}

\begin{proof}[Proof of Lemma~\ref{lem:QuantifyProb}]
First of all, note that if $\sigma_M(y)=0$, then $y \leq c_i$ for all $i=1,...,M$ and the probability equals one, and hence the identity holds. 

To show the result for $\sigma_M(y) \in (0,M)$, we need the joint density of $M$ order statistics, given by the constant $M!$~\cite[p.185]{Embrechts1997}, yielding
\begin{align*}
\Prob\left( U_{(i)}^M \leq \frac{c_i}{y}, \;\;\; \forall i \leq M \right) &= \int_0^{c_1/y} \int_{u_1}^{c_2/y} \cdots \int_{u_{\sigma_M(y)}-1}^{c_{\sigma_M(y)}/y} \int_{u_{\sigma_M(y)}}^{1} \cdots \int_{u_{M-1}}^{1}  M! \ du_{M} \cdots du_2 du_1 \\
&=  \frac{M!}{y^M} \int_0^{c_1} \int_{v_1}^{c_2} \cdots \int_{v_{\sigma_M(y)-1}}^{c_{\sigma_M(y)}} \int^{y}_{v_{\sigma_M(y)}} \cdots \int^{y}_{v_{M-1}}  1 \, dv_{M} \cdots dv_2 dv_1 \\
&=  \frac{M!}{y^M} \int_0^{c_1} \int_{v_1}^{c_2} \cdots \int_{v_{\sigma_M(y)-1}}^{c_{\sigma_M(y)}} \frac{(y-v_{\sigma_M(y)})^{M-\sigma_M(y)}}{(M-\sigma_M(y))!} \, dv_{\sigma_M(y)} \cdots dv_2 dv_1 \\
&= - \frac{M!}{y^M}   \beta_{\sigma_M(y)-1} \frac{(y-c_{\sigma_M(y)})^{M-\sigma_M(y)+1}}{(M-\sigma_M(y)+1)!}  \\
& \hspace{1.3cm}+\frac{M!}{y^M} \int_0^{c_1} \int_{v_1}^{c_2} \cdots \int_{v_{\sigma_M(y)-2}}^{c_{\sigma_M(y)-1}} \frac{(y-v_{\sigma_M(y)})^{M-\sigma_M(y)+1}}{(M-\sigma_M(y)+1)!} \,   dv_{\sigma_M(y)-1} \cdots dv_2 dv_1 \\
&= - \frac{M!}{y^M} \sum_{j=2}^{\sigma_M(y)} \beta_{j-1} \frac{(y-c_{j})^{M-j+1}}{(M-j+1)!}  +\frac{M!}{y^M} \int_0^{c_1} \frac{(y-v_1)^{M-1}}{(M-1)!} dv_1 \\
&= 1-\frac{M!}{y^M} \sum_{j=1}^{\sigma_M(y)} \beta_{j-1} \frac{(y-c_j)^{M-j+1}}{(M-j+1)!},
\end{align*}
where we used the change of variable $u_i=v_i/y$ for $i=1,...,M$ and then applied Lemma~\ref{lem:betaIdentity} in the Appendix multiple times.

For $\sigma_M(y)=M$, we observe that $y > c_i$ for all $i=1,...,M$ and which requires a separate analysis. Then,
\begin{align*}
\frac{y^M}{M!}\Prob\left( U_{(i)}^M \leq \frac{c_i}{y}, \;\;\; \forall i \leq M \right) &=  \int_0^{c_1} \int_{v_1}^{c_2} \cdots \int_{v_{M-1}}^{c_{M}} \, dv_{M} \cdots dv_2 dv_1 =  \beta_{M} \\
&= \sum_{j=0}^M \frac{y^j}{j!}-\sum_{j=1}^M \beta_{j-1} \sum_{l=0}^{M+1-j} \frac{(y-c_j)^l}{l!} \\
&= \frac{y^M}{M!} -\sum_{j=1}^M \beta_{j-1} \frac{(y-c_j)^{M+1-j}}{(M+1-j)!} + \sum_{j=0}^{M-1} \frac{y^j}{j!}-\sum_{j=1}^{M-1} \beta_{j-1} \sum_{l=0}^{M-j} \frac{(y-c_j)^l}{l!} -\beta_{M-1} \\
&= \frac{y^M}{M!} -\sum_{j=1}^M \beta_{j-1} \frac{(y-c_j)^{M+1-j}}{(M+1-j)!},
\end{align*}
where we applied Lemma~\ref{lem:BetaIdentity2} in the Appendix twice.
\end{proof}

\begin{lemma}
Conditions~(C) and (D) for perturbations~$\Delta(\cdot,\cdot)$ defined in Section~\ref{sec:ProofsOfC} are equivalent to the following condition: For every $\epsilon >0$ there exists a pair $(M_\epsilon,N_\epsilon) \in \mathbb{N} \times \mathbb{N}$ such that $|\Delta(i,N)|< \epsilon$ for all $N \geq N_\epsilon$ and all $M_\epsilon \leq i \leq k(N)$.
\label{lem:EquivalentCondition}
\end{lemma}

\begin{proof}
$(\Rightarrow)$ The boundedness of $\Delta(\cdot)$ is an immediate consequence of the boundedness of $\Delta(\cdot,\cdot)$. By definition of $\Delta(\cdot)$, we can pick a $\hat{N}_{i,\epsilon} \geq N_{\epsilon/2}$  for every $\epsilon>0$ and for all $i \geq M_{\epsilon/2}$ such that $|\Delta(i,N)-\Delta(i)|<\epsilon/2$ for all $N \geq \hat{N}_{i,\epsilon}$. Then,
\begin{align*}
|\Delta(i)| \leq |\Delta(i,\hat{N}_{i,\epsilon})|+\epsilon/2 < \epsilon
\end{align*}
for all $i \geq M_{\epsilon/2}$, showing that $\lim_{i \rightarrow \infty} \Delta(i)=0$. 

For condition~(D) to hold, suppose $\epsilon > 0$ and let $\tilde{N}_\epsilon \in \mathbb{N}$ be large enough such that $i(N) \geq M_\epsilon$ for all $N \geq \tilde{N}_\epsilon$ and $\tilde{N}_\epsilon \geq N_\epsilon$. Then, by assumption we obtain $|\Delta(i(N),N)|<\epsilon$ for all $N \geq \tilde{N}_\epsilon$.\\

$(\Leftarrow)$ If not, then $\exists \epsilon >0$ such that for every $(M_\epsilon,N_\epsilon) \in \mathbb{N} \times \mathbb{N}$ there exists an $i>M_\epsilon$ and $N \geq N_\epsilon$ such that $|\Delta(i,N)| \geq \epsilon$. In particular, if we choose $M_\epsilon= k(N_\epsilon)/2$, then there exists an $\epsilon >0$ such that for every $N_\epsilon \in \mathbb{N}$ there are a $ k(N_\epsilon)/2 \leq i \leq k(N_\epsilon)$ and $N \geq N_\epsilon$ such that $|\Delta(i,N)| \geq \epsilon$, contradicting condition~(D).
\end{proof}

\begin{proof}[Proof of Proposition~\ref{prop:AssFailedLines}]
It is known that the distribution function of a standard uniformly distributed random variable is contained in the maximum domain of attraction of a Weibull distribution~\cite[p.154]{Embrechts1997}:
\begin{align*}
\Prob\left(N(U_{(N)}^N-1) \leq  x\right) &= \Prob\left(U_{(N)}^N \leq 1+\frac{x}{N}\right) \longrightarrow \left\{ \begin{array}{ll}
e^x & x \leq 0,\\
1 & x>0,
\end{array}\right.
\end{align*}
as $N \rightarrow \infty$. 

\noindent
Then, by Theorem 4.2.8 of~\cite[p.201]{Embrechts1997}, the first $k$ order statistics converge in distribution to a particular distribution. More specifically, for every fixed $k \in \mathbb{N}$,
\begin{align*}
\left(N(U_{(N-i+1)}^N-1)\right)_{i=1,...,k} \overset{d}{\longrightarrow} \left(Y^{(i)}\right)_{i=1,...,k}
\end{align*}
as $N \rightarrow \infty$, where the joint density of $\left(Y^{(1)},Y^{(2)},...,Y^{(k)}\right)$ is given by 
\begin{align*}
\begin{array}{ll}
\psi_1(x_1,...,x_k)=e^{x_k} & x_k<...<x_1<0.
\end{array}
\end{align*}
\noindent
This observation is essential to determine the asymptotic exceedance probability, which we derive next.

First suppose that $c_{i,N}$ does not depend on $N$, i.e.~$c_{i,N}=c_i$ for all $N \in \mathbb{N}$. Then, the proof follows by induction. For $k=1$ the statement holds, since
\begin{align*}
\lim_{N \rightarrow \infty} \Prob\left(U_{(1)}^N \leq \frac{c_{1}}{N} \right) = \lim_{N \rightarrow \infty} 1-\left(1-\frac{c_{1}}{N}\right)^N = 1- e^{-c_{1}}.
\end{align*}
Suppose the statement holds for all integers strictly smaller than $k$. Then,
\begin{align*}
\lim_{N \rightarrow \infty} \Prob\left(U_{(i)}^N \leq \frac{c_{i}}{N} , \;\;\; \forall i \leq k \right)
&= \lim_{N \rightarrow \infty} \Prob\left(U_{(N-i+1)}^N > 1-\frac{c_{i}}{N} , \;\;\; \forall i \leq k \right) =  \Prob\left(Y^{(i)} > -c_{i} , \;\;\; \forall i \leq k \right) \\
&\hspace{-0.5cm}= \int_{-c_{1}}^0 \int_{-c_{2}}^{y_1} \cdots \int_{-c_{k}}^{y_{k-1}} e^{y_k} \; dy_k \cdots dy_1  \\
&\hspace{-0.5cm}= \int_{-c_{1}}^0 \int_{-c_{2}}^{y_1} \cdots \int_{-c_{k-1}}^{y_{k-1}} e^{y_{k-1}} \; dy_{k-1} \cdots dy_1  - e^{-c_{k}} \int_{-c_{1}}^0 \int_{-c_{2}}^{y_1} \cdots \int_{-c_{k-1}}^{y_{k-2}} dy_{k-1} \cdots dy_1 \\
&\hspace{-0.5cm}= 1-\sum_{j=1}^{k-1} \beta_{j-1} e^{-j}-e^{c_{k}}\beta_{k-1} = 1-\sum_{j=1}^k \beta_{j-1} e^{-j}.
\end{align*}
By induction, the statement thus holds for all $k \geq 1$.

Next, suppose $c_{i,N}$ does depend on $N$, i.e.~there is at least one $N \in \mathbb{N}$ such that $c_{i,N} \neq c_i$. Then,
\begin{align*}
& \Prob(\A \geq k) = \lim_{N \rightarrow \infty} \Prob\left( \U \leq \frac{c_{i,N}}{N}, \;\;\; \forall i\leq k\right) \overset{N \rightarrow \infty}{\longrightarrow} \Prob\left( N(U_{(N-i+1)}^N-1) \geq -c_{i}(1+o(1)), \; \forall i \leq k\right).
\end{align*}
Note that for every $\epsilon> 0$ (small enough) there exists a $N_0 \in \mathbb{N}$ such that for all $N\geq N_0$ and $1 \leq i \leq k$, 
\begin{align*}
\Prob(\A \geq k) \geq \Prob&\left( N(U_{(N-i+1)}^N-1) \geq -c_{i} +\epsilon , \forall i \leq k\right) 
\end{align*}
and
\begin{align*}
\Prob(\A \geq k) \leq \Prob\left( N(U_{(N-i+1)}^N-1) \geq -c_{i} - \epsilon , \forall i \leq k\right).
\end{align*}

\noindent
Write $V_1, V_2$ for the integration area of the upper bound and the lower bound respectively, and $V$ for the integration area corresponding to $c_{1},...,c_{k}$. Since $e^{x}<1$ for all $x<0$, it follows that
\begin{align*}
\limsup_{N \rightarrow \infty}   \, \Prob(\A \geq k) &= \int_{V_1} e^y \,dy \leq \int_{V} e^y \,dy+\int_{V_1 \backslash V} 1 \,dy \leq 1-\sum_{j=1}^k \beta_{j-1} e^{-j}+ k (c_{k} +\epsilon )^{k-1} \epsilon.
\end{align*}

\noindent
Similarly, for the lower bound we have
\begin{align*}
\liminf_{N \rightarrow \infty} \; \Prob(\A \geq k) &\geq \int_{V} e^y \,dy - \int_{V \backslash V_2} 1 \,dy \geq 1-\sum_{j=1}^k \beta_{j-1} e^{-j} - k (c_k)^{k-1} \epsilon.
\end{align*}
Letting $\epsilon \downarrow 0$ we obtain that both the upper bound and the lower bound converge to $1-\sum_{j=1}^k \beta_{j-1} e^{-j}$.
\end{proof}

\end{document}